\documentclass[12pt, english]{amsart}
\usepackage[
    a4paper,
    margin=3.5cm]{geometry}
\usepackage[utf8]{inputenc}
\usepackage{lipsum}
\usepackage{babel}
\usepackage[T2A,T1]{fontenc}
\usepackage{amsmath, amsthm, amssymb, amsfonts, mathtools, centernot, 
mathrsfs}
\usepackage{mathdots}
\usepackage{tikz,tikz-cd,enumerate, quiver-app}
\usepackage[colorlinks]{hyperref}
\usepackage{subcaption}
\usepackage[capitalise,nameinlink]{cleveref}
\usepackage{mathptmx}
\DeclareMathAlphabet{\mathcal}{OMS}{cmsy}{m}{n}

\newlength{\bibitemsep}
\newlength{\bibparskip}
\let\oldthebibliography\thebibliography
\renewcommand\thebibliography[1]{%
  \oldthebibliography{#1}%
  \setlength{\parskip}{\bibitemsep}%
  \setlength{\itemsep}{\bibparskip}%
}

\newtheorem{theorem}{Theorem}[section]

\newtheorem{corollary}[theorem]{Corollary}
\newtheorem{proposition}[theorem]{Proposition}
\newtheorem{lemma}[theorem]{Lemma}
\theoremstyle{definition}
\newtheorem{definition}[theorem]{Definition}
\newtheorem{example}[theorem]{Example}
\newtheorem{remark}[theorem]{Remark}

\usepackage{pagecolor}
\hypersetup{colorlinks,
citecolor=teal,
filecolor=green,
linkcolor=teal,
urlcolor=violet
}

\usepackage[
textwidth=3cm,
textsize=small,
colorinlistoftodos]
{todonotes}

\def\C{\mathbb{C}}
\newcommand{\Z}{\mathbb{Z}}
\DeclareMathOperator{\Hom}{Hom}
\DeclareMathOperator{\Ext}{Ext}

\DeclareMathOperator{\End}{End}

\DeclareMathOperator{\Fun}{Fun}

\DeclareMathOperator{\cof}{Cof}
\DeclareMathOperator{\tr}{tr}
\DeclareMathOperator{\res}{res}
\DeclareMathOperator{\ind}{ind}

\DeclareMathOperator{\proj}{proj}
\DeclareMathOperator{\mproj}{--\proj}
\DeclareMathOperator{\op}{op}
\DeclareMathOperator{\sg}{sg}

\DeclareMathOperator{\id}{id}
\DeclareMathOperator{\add}{add}
\DeclareMathOperator{\Mod}{Mod}
\DeclareMathOperator{\MMod}{--\Mod}
\DeclareMathOperator{\mmod}{--\mod}
\DeclareMathOperator{\Vect}{Vect}
\DeclareMathOperator{\perf}{perf}
\def\mod{\operatorname{mod}}
\DeclareMathOperator{\CM}{CM}
\DeclareMathOperator{\rad}{rad}

\DeclareMathOperator{\ev}{ev}
\DeclareMathOperator{\hecke}{{\mathcal{E}}}

\DeclareMathOperator{\Mack}{Mack}
\DeclareMathOperator{\CohMack}{CohMack}

\newcommand{\cohmu}{\mu^{coh}}
\DeclareMathOperator{\Ch}{\mathcal{C}{h}}
\DeclareMathOperator{\Ho}{Ho}

\newcommand{\F}{\mathbb{F}}
\newcommand{\und}[1]{\underline{#1}}
\newcommand{\undt}{\underline{\F}_2}

\newcommand{\undp}{\underline{\F}_p}

\begin{document}

\title{Equivariant $H\undp$-modules are wild}

\author[Grevstad]{Jacob Fjeld Grevstad}
\address{Department of Mathematical Sciences\\NTNU\\ 7491 Trondheim, Norway}
\email{jacob.f.grevstad@ntnu.no}

\author[May]{Clover May}
\address{Department of Mathematical Sciences\\NTNU\\ 7491 Trondheim, Norway}
\email{clover.may@ntnu.no}

\begin{abstract}
Let $k$ be an arbitrary field of characteristic $p$ and let $G$ be a finite group. We investigate the representation type, derived representation type, and singularity category of the $k$-linear cohomological Mackey algebra and the $k$-linear Mackey algebra.  We classify when the cohomological Mackey algebra is wild for $G$ a cyclic $p$-group. Furthermore, we show the cohomological Mackey algebra is derived wild whenever $G$ surjects onto a $p$-group of order more than two. We show the Mackey algebra is derived wild whenever $G$ is a nontrivial $p$-group.

Derived wildness has some immediate consequences in equivariant homotopy theory. In particular, for the constant Mackey functor $\und{k}$, the classification of compact modules over the $G$-equivariant Eilenberg--MacLane spectrum $H\und{k}$ is also wild whenever $G$ surjects onto a $p$-group of order more than two.  Thus, in contrast to recent work at the prime $2$ by Dugger, Hazel, and the second author in \cite{DHM}, no meaningful classification of compact $C_p$-equivariant $H\undp$-modules exists at odd primes.  For the Burnside Mackey functor $\und{A}_k$, there is no classification of compact $G$-equivariant $H\und{A}_k$-modules whenever $G$ is a nontrivial $p$-group.
\end{abstract}

\maketitle

\tableofcontents

\section{Introduction}

Mackey functors play an important role in equivariant homotopy theory as the analogue of stable homotopy groups and cohomology groups. For $G$ a finite group, every Mackey functor $M$ gives rise to a genuine equivariant Eilenberg--MacLane spectrum $HM$ representing $RO(G)$-graded equivariant cohomology with coefficients in $M$. The study of $RO(G)$-graded equivariant cohomology, graded by real representations of $G$, has played an important role in breakthroughs beyond equivariant homotopy theory including Hill, Hopkins, and Ravenel's landmark work on the Kervaire invariant one problem \cite{HHR16} and dos Santos, Florentino, and Orts' characterization of maximal and Galois-maximal real algebraic varieties \cite{dSFO25}. Equivariant Eilenberg--MacLane spectra are interesting in their own right, and are central to a large body of recent work from different perspectives including the study of Artin motives via tensor-triangular geometry, algebraic $K$-theory, and equivariant homotopy theory, see for example \cite{BG25, Fuh25, CW25, CV25, Pet24, Lev22}.
   
Due to a result of Schwede and Shipley \cite{SS03}, when $M$ is a commutative Mackey ring (also called a commutative Green functor), the homotopy category of $HM$-modules is equivalent to the derived category $\mathcal{D}(M)$. 
This equivalence of homotopy categories played a major role in the work of Dugger, Hazel, and the second author in \cite{DHM}, which gave a complete classification of compact modules over the $C_2$-equivariant Eilenberg--MacLane spectrum for the constant Mackey functor $\undt$.  It was shown in \cite{DHM} that there are just three families of indecomposable $H\undt$-modules: $H\undt$, $H\undt \wedge (S^m_a)_+$, and $\cof(\tau^r)$ for various $m \geq 0$ and $r \geq 1$. The classification of compact $H\undt$-modules was proved via a corresponding classification of indecomposable perfect complexes in the derived category $\mathcal{D}^{\perf}(\undt)$.  Moreover, this classification was used in \cite{DHM} to give a new proof of the second author's structure theorem for $RO(C_2)$-graded cohomology \cite{CM20}, implying the only indecomposables coming from the (co)homology of $C_2$-CW complexes are $H\undt$ and $H\undt \wedge (S^m_a)_+$. 
  
One might hope to generalize these results to a classification of compact $H \undp$-modules for other finite groups.  However, this is not possible. The category of $H\undp$-modules is much richer in general. In this paper we use quiver representations to show $\undp$ is derived wild for most groups.  Thus a complete classification is unattainable, both for equivariant $H\undp$-modules and for perfect complexes of $\undp$-modules. Here we mean \emph{wild} a in a technical sense; any complete classification would necessarily include a classification of indecomposable modules of every finite dimensional $\F_p$-algebra. Although $H\undp$-modules cannot be classified, the question of whether a meaningful structure theorem for $RO(G)$-graded cohomology exists at odd primes remains open.

Turning to an arbitrary field $k$ of characteristic $p$, we study the derived category of $\und{k}$-modules, also known as cohomological Mackey functors. In \cite{TW95}, Thévenaz and Webb classified precisely the groups for which Mackey functors and cohomological Mackey functors have finite representation type.  We investigate their representation types further, as well as derived representation types, motivated by equivariant homotopy theory. In order to study (derived) modules over the constant Mackey functor $\und{k}$, we study the equivalent category of (derived) modules over the cohomological Mackey algebra $\cohmu_k(G)$. Beyond \cite{TW95}, the cohomological Mackey algebra has been studied from an algebraic perspective by a number of people, for example in \cite{Bouc11, BS15, BSW17, Lin18}.
  
Representation type is more complicated over arbitrary fields, but for algebras over an algebraically closed field, the famous tame-wild dichotomy due to Drozd \cite{Dro77,CB88} separates algebras into tame or wild representation type (but not both).  Morally, the indecomposable representations of tame algebras are ``describable'', while  the indecomposable representations of wild algebras are hopeless to classify. We classify the representation type of the cohomological Mackey algebra $\cohmu_k(G)$ for cyclic $p$-groups over algebraically closed fields. 

\begin{theorem}[\emph{c.f.\ }\cref{thm:alg_closed_cohomological_rep_type}]
  Let $k = \bar{k}$ be an algebraically closed field of characteristic $p$.  The cohomological Mackey algebra $\cohmu_k(C_{p^m})$ has
  \begin{enumerate}
    \item finite representation type if $m\leq 1$,
    \item tame (and infinite) representation type if $p^m=4$, and
    \item wild representation type otherwise.
  \end{enumerate}
\end{theorem}

However, due to our interest in equivariant homotopy, much of our focus is on non-algebraically closed fields.  Extending the definition of wild representation type to arbitrary fields in the obvious way, we show the cohomological Mackey algebra, as well as the Mackey algebra, in prime characteristic are wild in the same case as above, whenever $G=C_{p^m}$ has $m \geq 2$ and $p^m \neq 4$ (see \cref{thm:arbitrary_field_cyclic_p_group_cohomological_rep_type} and \cref{cor:mackey_algebra_wild_cyclic}).

We then turn to derived representation type and show the cohomological Mackey algebra is derived wild for most finite groups.  Over a field of characteristic $p$, we first show $\cohmu_k(G)$ is derived wild for $p$-groups of order more than two.  That is, we show the perfect derived category $\mathcal{D}^{\perf}(\und{k})$ is wild.  The proof uses the derived wildness of certain group algebras $kG$ over (possibly finite) fields, which we explore in \cref{sec:derived_wild_groups}. We then generalize to larger groups and obtain the following result.

\begin{theorem}[\emph{c.f.\ }\cref{thm:cohomological_derived_wild_surjection}]
 Let $k$ be a field of characteristic $p$. The cohomological Mackey algebra  $\cohmu_k(G)$ is derived wild whenever $G$ surjects onto a $p$-group of order more than two.
\end{theorem}

This implies there is no meaningful classification of modules over the  Eilenberg--MacLane spectrum for the constant Mackey functor $\und{k}$.

\begin{theorem}[\emph{c.f.\ }\cref{thm:cohomological_derived_wild_surjection_topological}]
 Let $k$ be a field of characteristic $p$. The category of compact objects in the homotopy category of $G$-equivariant $H\und{k}$-modules is wild whenever $G$ surjects onto a $p$-group of order more than two.
\end{theorem}

For $G=C_{p^n}$ or $G=C_p \times C_p$, this follows directly from the proof that $kG$ is derived wild. Thus, in a precise sense, the wildness of $H\und{k}$-modules comes from $Fun(BG,k\MMod)$, $k$-modules with a $G$-action.  

Over an algebraically closed field, there is a derived version of the tame-wild dichotomy due to Bekkert and Drozd \cite{BD03}. In the algebraically closed case, we classify precisely when the cohomological Mackey algebra is derived wild by reducing to $p$-Sylow subgroups.

\begin{theorem}[\emph{c.f.\ }\cref{thm:alg_closed_derived_type_cohomological_sylow}]
  If $k=\bar{k}$ is an algebraically closed field of characteristic $p$ and $G$ is a finite group, then $\cohmu_k(G)$ is derived wild if and only if the $p$-Sylow subgroup of $G$ has order more than two.
\end{theorem}

From a homotopy theory perspective, the huge leap moving from $C_2$ to larger groups is surprising. What makes $C_2$ nice is that the cohomological Mackey algebra $\cohmu_k(C_2)$ is a gentle algebra, controlled by the very small bound quiver with a single relation
 \begin{center}
    \begin{tikzcd}
      \bullet \ar[r, bend left, "b"] & \bullet \ar[l, bend left, "a"] & ba=0.
    \end{tikzcd}
  \end{center}
Gentle algebras (over algebraically closed fields) are derived tame, and hence the perfect complexes are well understood. On the other hand, even for $C_3$, the cohomological Mackey algebra $\cohmu_k(C_3)$ in characteristic $3$ is no longer gentle.  This allows the perfect derived category to be much richer.

At the prime $2$, we show the Mackey algebra $\mu_k(C_2)$ is derived wild. Thus there is no meaningful classification of modules over the Eilenberg--MacLane spectrum for the mod 2 Burnside Mackey functor $\und{A}\otimes \F_2$. More generally, for any finite $p$-group we prove the following.

\begin{theorem}[\emph{c.f.\ }\cref{thm:derived_wild_more_p-groups_mackey}]
  Let $k$ be a field of characteristic $p$ and let $G$ be a nontrivial finite $p$-group. Then the Mackey algebra $\mu_k(G)$ is derived wild.
\end{theorem}

Hence there is also no meaningful classification of $H\und{A}_k$-modules.

\begin{theorem}[\emph{c.f.\ }\cref{thm:derived_wild_more_p-groups_mackey_topological}]
  Let $k$ be a field of characteristic $p$.  The category of compact objects in the homotopy category of $G$-equivariant $H\und{A}_k$-modules is wild whenever $G$ is a nontrivial finite $p$-group.
\end{theorem}

We go on to compute some examples of the singularity category for the cohomological Mackey algebra.  The singularity category is the Verdier quotient of the bounded derived category by the perfect derived category, 
\[
\mathcal{D}_{\sg}(\cohmu_k(G)) \cong \mathcal{D}^{b}(\cohmu_k(G))/\mathcal{D}^{\perf}(\cohmu_k(G)).
\]
For $G=C_2$ in characteristic $2$, this is trivial as $\mathcal{D}^{b}(\cohmu_k(C_2))
= \mathcal{D}^{\perf}(\cohmu_k(C_2))$. In general the singularity category is nontrivial, for example even for $G=C_3$ in characteristic $3$.  We compute the singularity category for cyclic $p$-groups $G =C_{p^m}$ and the Klein four group $G=C_2 \times C_2$.

\bigskip

With the hope that our results are of interest both to equivariant homotopy theorists and representation theorists, we provide significant background on representation type, derived representation type, and Mackey functors in equivariant homotopy theory.  Along the way, we present proofs of a number of well-known results in representation theory in order to fill some gaps in the literature for arbitrary fields.

  \subsubsection{Organization of the paper}
In \cref{sec:preliminaries}, we give an extensive introduction to quiver representations.  We also introduce Mackey functors, and both the Mackey algebra and the cohomological Mackey algebra. In addition, we provide some background on equivariant homotopy theory.  In \cref{sec:reptype}, we give an overview of the representation type of algebras, with particular emphasis on arbitrary fields.  In \cref{sec:rep_type_cohom_mackey}, we investigate the representation type of the cohomological Mackey algebra, both over algebraically closed and arbitrary fields.  In \cref{sec:derived_wild_groups}, we introduce the notion of derived representation type and fill in some details showing certain group algebras are derived wild over (not necessarily algebraically closed) fields of prime characteristic. Then in \cref{sec:derived-wildnes-cohmackey}, we apply these results to classify the derived representation type of the cohomological Mackey algebra.  In \cref{sec:mackey_algebra}, we consider the representation type and derived representation type of the Mackey algebra. Finally, in \cref{sec:singularity_category}, we compute some examples of the singularity category for some cohomological Mackey algebras.

\subsubsection{Acknowledgements}
The authors would like to thank Torgeir Aambø and Endre Rundsveen for organizing the PITA seminar at NTNU, which is where this collaboration began. We would also like to thank Tim Campion, Jens Niklas Eberhardt, Drew Heard, Mike Hill, Sondre Kvamme, Julian Külshammer, Laertis Vaso, and Peter Webb for helpful conversations. Special thanks to John Greenlees for encouraging and insightful conversations, and for comments on an earlier draft.

The first author acknowledges support from The Research Council of Norway (RCN 301375).  The second author is supported by grant number TMS2020TMT02 from the Trond Mohn Foundation.  The second author would also like to thank the Isaac Newton Institute for Mathematical Sciences, Cambridge, for support and hospitality during the programme Equivariant Homotopy Theory in Context where some of the work on this paper was undertaken. This work was also supported by EPSRC grants EP/Z000580/1 and EP/R014604/1.  

\subsection{Notation and conventions}
In this paper $G$ will denote a finite group.  We write $C_p$ for the cyclic group of order $p$, with $p$ a prime.  
Throughout, $k$ will denote a field.  Often $k$ will be a field of prime  characteristic $p$, e.g.\ the field $\F_p$ with $p$ elements. The algebraic closure of $k$ is denoted $\bar{k}$.

By algebra over $k$ we mean an associative unital algebra, and by module we always mean a left module unless otherwise specified. For a $k$-algebra $\Lambda$ we write $\Lambda\MMod$ for the category of left $\Lambda$-modules and $\Lambda \mmod$ for the category of finite dimensional  $\Lambda$-modules. 
If $\Lambda$ and $\Gamma$ are $k$-algebras and $M$ is a $\Lambda$--$\Gamma$-bimodule we always assume $k$ acts symmetrically on $M$. This is equivalent to $M$ being a $\Lambda \otimes_k \Gamma^{\, \op}$-module.

\section{Preliminaries}\label[section]{sec:preliminaries}

\subsection{Quivers}
Quivers and their representations are combinatorial tools used to study representation theory of algebras.  In this section we introduce the basic definitions.  For more details see for example \cite{AssemCoelho} or \cite{AssemSkowronskiSimson1}.

A quiver is a directed multigraph made up of vertices and arrows, allowing loops and/or multiple arrows between vertices.  A common definition with more explicit notation is given below.

\begin{definition}[Quiver]
  A \emph{quiver} is a tuple $Q = (Q_0, Q_1, s, t)$, where $Q_0$ and $Q_1$ are (typically finite) sets, and $s$ and $t$ are functions $Q_1 \to Q_0$. The elements of $Q_0$ are called the vertices of the quiver $Q$ and the elements of $Q_1$ are called the arrows. For an arrow $\alpha$, we call $s(\alpha)$ and $t(\alpha)$ the source and target of $\alpha$ respectively.

  A \emph{morphism} of quivers $f \colon Q \to Q'$ is a pair of functions $f = (f_0,f_1)$ sending vertices to vertices and arrows to arrows, commuting with source and target. 
  \end{definition}
  
  We usually illustrate a quiver by drawing the vertices as dots and drawing an arrow from $s(\alpha)$ to $t(\alpha)$ for each $\alpha$ in $Q_1$.
\begin{example}
  Below are some examples of quivers. 
  
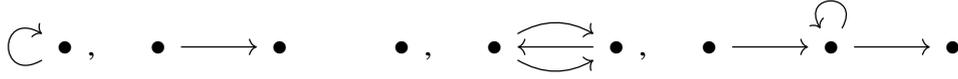
\begin{center}
    \begin{tikzcd}
      \bullet \arrow[loop left, out=210, in=150, looseness=5]
    \end{tikzcd}, \quad
    \begin{tikzcd}
      \bullet \ar[r] & \bullet & \bullet
    \end{tikzcd}, \quad
    \begin{tikzcd}
      \bullet \ar[r, bend left] \ar[r, bend right] & \bullet \ar[l]
    \end{tikzcd}, \quad
    \begin{tikzcd}
      \bullet \ar[r]  & \bullet \arrow[loop left, out=60,in=120, looseness=5] \ar[r] & \bullet
    \end{tikzcd}
  \end{center}
\end{example}

\begin{definition}[Paths]
  A \emph{path} in a quiver $Q$ is either a trivial path at a vertex $i \in Q_0$, denoted $e_i$, or it is a sequence of arrows $\alpha_1 \alpha_2 \cdots \alpha_n$ such that $s(\alpha_j) = t(\alpha_{j+1})$ for $j=1, 2, \cdots, n-1$.  We extend $s$ and $t$ to the set of all paths by defining $s(\alpha_1 \alpha_2 \cdots \alpha_n) = s(\alpha_n)$, $t(\alpha_1 \alpha_2 \cdots \alpha_n) = t(\alpha_1)$, and $s(e_i) = t(e_i) = i$.

  The \emph{length} of a path is the number of arrows in the path, and is $0$ for a trivial path.
  
  We say that two paths $\rho_1$ and $\rho_2$ are composable if $s(\rho_1) = t(\rho_2)$. The composition of two composable paths is simply their concatenation if they are both nontrivial, and it is the other path if one of them is trivial.  We write the composition $\rho_1 \rho_2$ as in function composition, $\rho_1 \circ \rho_2$. The reader is advised that the notation in the literature varies and much of the literature on quivers in representation theory uses the opposite convention. 
\end{definition}

\begin{remark}
  A quiver $Q$ defines a category with objects given by the vertices and morphisms given by paths. By abuse of notation this is also written $Q$.
\end{remark}

Using this categorical perspective we can define quiver representations.

\begin{definition}[Quiver representation]
  Fix a field $k$ and let $\Vect_k$ denote the category of $k$-vector spaces.  A \emph{representation} of a quiver $Q$ is a functor
  \[
  F \colon Q \to \Vect_{k}.
  \]

That is, a \emph{quiver representation} for $Q$ consists of a $k$-vector space $V_i$ for each vertex $i \in Q_0$ and a linear map $f_{\alpha} \colon V_{s(\alpha)} \to V_{t(\alpha)}$ for each arrow $\alpha \in Q_1$, such that composable arrows give rise to composable morphisms.  To each trivial path $e_i$, one assigns the identity map $\id_{V_i}$.
\end{definition}

\begin{remark}
  There are a priori no relations between the linear maps in a quiver representation.  For example to impose commutativity, one must introduce such relations.
\end{remark}

  \begin{definition}[Bound quiver]
  Fixing a field $k$ and a quiver $Q$, a \emph{relation} is a linear combination of paths of length at least two with the same source and target. A quiver together with a set of relations $\rho$ is called a \emph{bound quiver} or \emph{quiver with relations}.
\end{definition}

As in the unbound case, we can work with representations of bound quivers.

\begin{definition}(Bound quiver representation)
  A \emph{representation} of a bound quiver $(Q, \rho)$ is a functor 
  \[
  F \colon (Q,\rho) \to \Vect_{k},
  \]
  where the linear maps satisfy the relations in $\rho$.
\end{definition}

Quivers and bound quivers give rise to algebras.

\begin{definition}[Path algebra]
  Given a field $k$ and a quiver $Q$, let $kQ$ be the free vector space generated by the paths in $Q$. Define a bilinear product on $kQ$ by setting the product of paths to be their composition if they are composable and $0$ otherwise. Then $kQ$ becomes an associative $k$-algebra known as the \emph{path algebra} of $Q$.
  
  In the path algebra, each trivial path $e_i$ is an idempotent.  If the quiver has finitely many vertices, then $\sum_{i \in Q} e_i$ is the multiplicative identity in $kQ$.

  If $(Q, \rho)$ is a bound quiver, the path algebra of the bound quiver is defined to be the quotient $kQ/I$, where $I = (\rho)$ is the two-sided ideal generated by a set of relations.\footnote{It is common in the literature to denote such a bound quiver as $(Q, I)$.}  If $Q$ is a finite quiver and there is a number $n$ such that all paths of length at least $n$ are contained in $(\rho)$, then $kQ/(\rho)$ is called an \emph{admissible path algebra}. 
\end{definition}

\begin{remark}
  An admissible path algebra $kQ/(\rho)$ is a finite dimensional $k$-algebra.
\end{remark}

\begin{example}
  The path algebra of the quiver
  \begin{center}
    \begin{tikzcd}
      \bullet \arrow[loop left, out=210, in=150, looseness=5,"x"]
    \end{tikzcd}
  \end{center}
  is the polynomial algebra $k[x]$.
\end{example}

\begin{example}\label[example]{ex:freealgebraxy}
  The path algebra of the quiver
  \begin{center}
    \begin{tikzcd}
      \bullet \arrow[loop left, out=210, in=150, looseness=5,"x"] \arrow[loop left, out=-30, in=30, looseness=5,"y",swap]
    \end{tikzcd}
  \end{center}
  is the free algebra $k\langle x,y\rangle$.
\end{example}

\begin{example}
  The path algebra of the bound quiver
  \begin{center}
    \begin{tikzcd}
      \bullet \arrow[loop left, out=210, in=150, looseness=5,"x"]
    \end{tikzcd}
  \end{center}
  with relations $\rho = \{ x^5 \}$ is the truncated polynomial algebra $k[x]/(x^5)$.
\end{example}

\begin{example}
  The path algebra of the bound quiver
  \begin{center}
    \begin{tikzcd}
      \bullet \arrow[loop left, out=210, in=150, looseness=5,"x"] \arrow[loop left, out=-30, in=30, looseness=5,"y",swap]
    \end{tikzcd}
  \end{center}
  with relations $\rho = \{ xy-yx \}$ is the polynomial algebra $k[x,y]$.
\end{example}

One can study representations of path algebras by studying quiver representations.

\begin{theorem}\label[theorem]{Thm:QuiverRepsAreModules}
  The category of $k$-linear representations of $(Q,\rho)$ is equivalent to the category of left $kQ/(\rho)$-modules.
\end{theorem}

Path algebras of bound quivers are a rich source of algebras with nice combinatorial properties.  Moreover, perhaps one of the most compelling reasons to study quiver representations is that (over algebraically closed fields) almost every finite dimensional algebra is the path algebra of a bound quiver.  To state this precisely, as in \cite[Theorem~I.2.13]{AssemCoelho} for example, we need a minor technical condition, requiring the algebra to be basic\footnote{In particular, this ensures that the quotient by the Jacobson radical $A/J(A)$ is isomorphic to $\Pi_{i=1}^n k$.}.

\begin{theorem}\label[theorem]{Thm:QuiverExists}
  Let $k = \bar{k}$ be an algebraically closed field and let $A$ be any basic finite dimensional $k$-algebra. There exists a bound quiver $(Q, \rho)$ with admissible relations such that $A \cong kQ/(\rho)$ as algebras. Further, the quiver $Q$, known as the Gabriel quiver of $A$, is uniquely determined.
\end{theorem}

Thus, putting \cref{Thm:QuiverRepsAreModules} and \cref{Thm:QuiverExists} together, to study the representation theory of an algebra over $k=\bar{k}$, we may look at representations of its bound quiver. Every finite dimensional algebra is Morita equivalent to a basic algebra, so the requirement that the algebra be basic is indeed a minor condition.  In terms of representation theory, finite dimensional $\bar{k}$-algebras are equivalent to path algebras of bound quivers.

Quivers easily capture certain information about representations, for example one can immediately identify the simple modules.

\begin{proposition}
Let $A = kQ/(\rho)$ be an admissible path algebra, and let $S(a)$ denote the quiver representation with $k$ at vertex $a \in Q_0$ and $0$ everywhere else.  The set $\{S(a) \mid a \in Q_0 \}$ is a complete set of representatives of the isomorphism classes of simple $A$-modules. 
\end{proposition}

\begin{example}
  The quiver $A_n$
  \begin{center}
  \begin{tikzcd}
    1 \arrow[r] & 2 \arrow[r] & \cdots \arrow[r] & i \arrow[r] & \cdots \arrow[r] & n
  \end{tikzcd}
\end{center}
has simples $\{S_1, S_2, \dots, S_n\}$ with $S_i$ given by $k$ in the $i$-th spot an zero elsewhere
\begin{center}
  \begin{tikzcd}
   0 \arrow[r] & 0 \arrow[r] & \cdots \arrow[r] & k \arrow[r] & \cdots \arrow[r] & 0.
  \end{tikzcd}
\end{center}
\end{example}

\begin{remark}
  More generally, given a module $M$ over a path algebra $\Lambda = kQ/(\rho)$ we illustrate the module by labelling each vertex $i$ with the vector space $e_i M$ and each arrow $\alpha$ by the linear map $e_{s(\alpha)}M \to e_{t(\alpha)} M$ given by multiplication by $\alpha$. This information completely determines $M$.
  \end{remark}

Quivers also make it easy to identify indecomposable projective modules.

\begin{proposition}\label[proposition]{prop:IndecompProjectives}
  Let $\Lambda = kQ/(\rho)$ be an admissible path algebra. Then the trivial paths $\{e_1,\dots,e_n\}$ form a complete set of primitive orthogonal idempotents. The path algebra $\Lambda$ decomposes into indecomposable projective summands
  \[
  \Lambda = P_1 \oplus \cdots \oplus P_n,
  \]
  with each $P_i$ equal to $\Lambda e_i$, the linear span of all paths in the path algebra starting at $i$. Moreover, $\{P_1, \dots, P_n\}$ form a complete set of representatives of isomorphism classes of indecomposable projective $\Lambda$-modules. 
\end{proposition}

\bigskip

This paper is primarily concerned with algebras that come from Mackey functors due to their prevalence in equivariant homotopy theory.

\subsection{Mackey functors}
In equivariant stable homotopy theory, homotopy takes values in Mackey functors, rather than abelian groups. Mackey functors were first defined by Dress and Green \cite{Dre71, Dre73, Gre71}.  Three equivalent definitions can be found in \cite{TW95}.  We also direct the reader to \cite{Web00} and \cite{Shu10} for more details. 
A Mackey functor is a collection of abelian groups with restriction, transfer, and conjugation maps satisfying a number of relations. 

\begin{definition}[Mackey functor]\label[definition]{def:MackeyFunctor}
  A \emph{$G$-Mackey functor} $M$ is a function\footnote{One will often find in the literature $M$ written as a function with values determined by subgroups, replacing $M(G/H)$ with $M(H)$.  However, the notation $M(G/H)$ is consistent with the other two equivalent definitions in \cite{TW95}.  Both of the others define a Mackey functor with values determined by finite $G$-sets, which decompose into orbits.}
  \[
  M \colon \{G/H \mid H\leq G \} \to \mathbf{Ab}
  \]
  with morphisms called \emph{restrictions, transfers, and conjugation maps}
  \begin{align*}
    \res_K^H &\colon M(G/H) \to M(G/K)\\
    \tr_K^H &\colon M(G/K) \to M(G/H)\\
    \ c_g &\colon M(G/H) \to M(G/gHg^{-1})
  \end{align*}
  for all subgroups $K \leq H$ and for all $g \in G$, such that
\begin{enumerate}
  \item $\tr_H^H$, $\res_H^H$, $c_h \colon M(G/H) \to M(G/H)$ are the identity morphisms for all subgroups $H \leq G$ and $h \in H$,
  \item $\res_J^K \res_K^H = \res_J^H$ and $\tr_K^H \tr_J^K = \tr_J^H$ for all subgroups $J \leq K \leq H$,
  \item $c_g c_h = c_{gh}$ for all $g,h \in G$,
  \item $\res_{gKg^{-1}}^{gHg^{-1}} c_g = c_g \res_K^H$ and $\tr_{gKg^{-1}}^{gHg^{-1}} c_g = c_g \tr_K^H$ for all subgroups $K \leq H$ and $g \in G$, and
  \item $\res_J^H \tr_K^H$ satisfies the double coset formula, i.e.\ for all subgroups $J, K \leq H$,
\[
  \res_J^H \tr_K^H = \sum_{x \in [J\backslash H/ K]} \tr_{J\, \cap\, xKx^{-1}}^J c_x \res_{x^{-1}Jx\, \cap\, K}^K.
\]
\end{enumerate}
\end{definition}

For a fixed finite group $G$, the collection of $G$-Mackey functors form a category $\Mack(G)$.  This category is abelian and symmetric monoidal via the box product $\Box$.  The unit with respect to the box product is the Burnside Mackey functor $\und{A}$, with value at $G/H$ given by the Burnside ring $A(H)$. See for example \cite{Lewis88} or \cite{Shu10}.

Fixing a commutative ring $R$, one may consider Mackey functors valued in $R$-modules.  
We denote the category of such $R$-linear Mackey functors by $\Mack_R(G)$ as in \cite{TW95}.  The unit in $\Mack_R(G)$ is $\und{A}_R$, formed by tensoring the Burnside Mackey functor with the ring object-wise $\und{A} \otimes R$.  Taking $R = \Z$, we recover the original category $\Mack(G) = \Mack_\Z(G)$.  In this paper, we will be concerned with $\Mack_k(G)$ where $k$ is a field (typically of characteristic $p$).

For groups with a small number of subgroups, it is convenient to describe the data of a Mackey functor by its Lewis diagram.

\begin{example}
  Let $p$ be a prime and let $G=C_p$, the cyclic group of order $p$.  A $C_p$-Mackey functor $M$ has a Lewis diagram of the following form,
\[
\begin{tikzcd}
  M(C_p/C_p) \arrow[d, "\res_e^{C_p}"left, swap, shift right=1.5] \\ M(C_p/e)  \arrow[u, "\tr_e^{C_p}"right, swap, shift right=1.5] \arrow[loop, distance=25, out=300,in=240, "\gamma"]
\end{tikzcd}
\]
where $\gamma$ is a generator of the Weyl group\footnote{The Weyl group in equivariant homotopy theory is the quotient of the normalizer of a subgroup $W_G H = N_G H/ H$ and is isomorphic to the equivariant automorphisms of the $G$-set $G/H$.} $W_{C_p}e \cong C_p/e \cong C_p = \langle \gamma \rangle$ and the maps satisfy the following relations:
\begin{enumerate}
  \item $\gamma \circ \res_e^{C_p} = \res_e^{C_p}$,
  \item $\tr_e^{C_p} \circ \gamma = \tr_e^{C_p}$,
  \item $\gamma^{\,p} = \id$, and
  \item $\res_e^{C_p} \circ \tr_e^{C_p} = \Sigma_{i=0}^{p-1} \gamma^{\,i}$
\end{enumerate}
\end{example}

The Weyl group action corresponds to the conjugation $c_\gamma$ in \cref{def:MackeyFunctor}. The last relation is implied by the double coset formula.

\begin{example}\label[example]{ex:BurnsideMackey}
  For $G=C_p$, The Burnside Mackey functor $\und{A}$ has the following Lewis diagram.
\[
\begin{tikzcd}
  \Z \oplus \Z \arrow[d, "\begin{psmallmatrix}1 \ \  p \end{psmallmatrix}"left, swap, shift right=1.25] \\ \Z  \arrow[u, "\begin{psmallmatrix} 0 \\ 1 \rule{0pt}{7pt}\end{psmallmatrix}"right, swap, shift right=1.25] \arrow[loop, distance=25, out=300,in=240, "1"]
\end{tikzcd}
\]
\end{example}

\begin{definition}[Fixed point Mackey functor]
  For $G$ any finite group and $V$ an $R G$-module, the \emph{fixed point Mackey functor} $FP_V$ is defined by $FP_V(H) = V^H$, the $H$-fixed points in $V$.  Restriction $\res_K^H \colon V^H \to V^K$ is given by inclusion, transfer $\tr_K^H \colon V^K \to V^H$ is the relative trace map $x \mapsto \Sigma_{h \in [H/K]}h \cdot x$, and conjugation $c_g$ is given by multiplication by $g$.
\end{definition}

The category of abelian groups embeds into $\Mack_{\Z}(G)$ as the constant Mackey functors.

\begin{definition}[Constant Mackey functor]
  Fix a finite group $G$ and an abelian group $T$.  The \emph{constant $G$-Mackey functor} $\und{T}$ sends every orbit $G/H$ to $T$ and assigns the identity to all restriction maps and all conjugation maps.\footnote{The transfer maps, given by relative trace maps, are determined by this data and the relations imposed by the double coset formula.}  Equivalently, the constant $G$-Mackey functor $\und{T}$ is the fixed point Mackey functor of the trivial $\Z G$-module $T$, i.e.\ $\und{T} = FP_T$.
\end{definition}

When $T$ is a ring, the constant Mackey functor $\und{T}$ is a Mackey ring (also called a Green functor).  See \cite{Lewis88}, \cite{TW95}, or \cite{Shu10} for more about Mackey rings.

\begin{example}
  The constant $G$-Mackey functor $\und{\Z}$ assigns to every orbit the ring $\Z$. 
  
  For $G=C_p$, the constant $C_p$-Mackey functor $\und{\Z}$ has the following Lewis diagram.
\[
\begin{tikzcd}
  \Z \arrow[d, "1"left, swap, shift right=1.25] \\ \Z  \arrow[u, "p"right, swap, shift right=1.25] \arrow[loop, distance=25, out=300,in=240, "1"]
\end{tikzcd}
\]
\end{example}

Much of our work here is motivated by our interest in the constant Mackey functor $\undp$.

\begin{example}\label[example]{ex:CpFpMackey}
  Let $k$ be a field of characteristic $p$.  The constant $G$-Mackey functor $\und{k}$ assigns to every orbit the field $k$.  For example, the constant $G$-Mackey functor $\undp$ assigns to every orbit the field $\F_p$. 
  
  For $G=C_p$, the constant $C_p$-Mackey functor $\undp$ has the following Lewis diagram.
\[
\begin{tikzcd}
  \F_p \arrow[d, "1"left, swap, shift right=1.25] \\ \F_p  \arrow[u, "0"right, swap, shift right=1.25] \arrow[loop, distance=25, out=300,in=240, "1"]
\end{tikzcd}
\]
\end{example}

\begin{example}\label[example]{ex:ZModMackey}
Using the box product, a $C_p$-Mackey functor $M$ is a $\und{\Z}$-module if it has Lewis diagram
\[
\begin{tikzcd}
  M(C_p/C_p) \arrow[d, "\res_e^{C_p}"left, swap, shift right=1.5] \\ M(C_p/e)  \arrow[u, "\tr_e^{C_p}"right, swap, shift right=1.5] \arrow[loop, distance=25, out=300,in=240, "\gamma"]
\end{tikzcd}
\]
and maps satisfying the usual relations
\begin{enumerate}
  \item $\gamma \circ \res_e^{C_p} = \res_e^{C_p}$,
  \item $\tr_e^{C_p} \circ \gamma = \tr_e^{C_p}$,
  \item $\gamma^p = \id$, and
  \item $\res_e^{C_p} \circ \tr_e^{C_p} = \Sigma_{i=0}^{p-1} \gamma^i$,
\end{enumerate}
as well as the relation
\begin{enumerate}
  \item[(5)] $\tr_e^{C_p} \circ \res_e^{C_p} = p$.
\end{enumerate}
\end{example}

Tensoring with $\F_p$, we have the following characterization of $\und{\F}_p$-modules.

\begin{example}\label[example]{ex:FpModMackey}
  A $C_p$-Mackey functor $M$ is an $\und{\F}_p$-module if it has a Lewis diagram 
  \[
  \begin{tikzcd}
    M(C_p/C_p) \arrow[d, "\res_e^{C_p}"left, swap, shift right=1.5] \\ M(C_p/e)  \arrow[u, "\tr_e^{C_p}"right, swap, shift right=1.5] \arrow[loop, distance=25, out=300,in=240, "\gamma"]
  \end{tikzcd}
  \]
  valued in $\Vect_{\F_p}$ and maps satisfying 
  \begin{enumerate}
    \item $\gamma \circ \res_e^{C_p} = \res_e^{C_p}$,
    \item $\tr_e^{C_p} \circ \gamma = \tr_e^{C_p}$,
    \item $\gamma^p = \id$
    \item $\res_e^{C_p} \circ \tr_e^{C_p} = \Sigma_{i=0}^{p-1} \gamma^i$, and
    \item $\tr_e^{C_p} \circ \res_e^{C_p} = 0$.
  \end{enumerate}
  \end{example}

Fix a finite group $G$ and a commutative ring $R$.  There is a subcategory of $\Mack_R(G)$ of cohomological Mackey functors, denoted $\CohMack_R(G)$.  

\begin{definition}[Cohomological Mackey functor]
A Mackey functor is \emph{cohomological} if it satisfies the additional relations that whenever $K \leq H \leq G$, restriction followed by transfer is multiplication by the index: \[
\tr_H^K \circ \res_H^K = [K:H].
\]
\end{definition}
In particular, all constant Mackey functors are cohomological.  In fact, all fixed point Mackey functors are cohomological.  Moreover, we can characterize cohomological Mackey functors completely.  As shown in \cite[Prop~16.3]{TW95}, cohomological Mackey functors are equivalent to $\und{\Z}$-modules
 \[
 \CohMack_\Z(G) \cong \und{\Z}\MMod,
 \]
 and over a commutative ring $R$
 \[
  \CohMack_R(G) \cong \und{R}\MMod.
  \]
 Taking $R = \F_p$, we have
\[
\CohMack_{\F_p}(G) \cong \und{\F}_p\MMod.
 \]

\bigskip
The connection between Mackey functors and quivers is key to the results of this paper.
\subsection{Mackey functors as quiver representations} As observed by Thévenaz and Webb in \cite{TW95}, a Mackey functor is a representation of a bound quiver.  We include some examples that we will return to in \cref{sec:rep_type_cohom_mackey} and \cref{sec:derived-wildnes-cohmackey}.

\begin{example}\label[example]{ex:C2constantquiver}
  Let $G=C_2$.  An $\und{\F}_2$-module is a representation of the following bound quiver $Q$
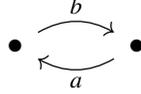
  \begin{center}
    \begin{tikzcd}
      \bullet \ar[r, bend left, "b"] & \bullet \ar[l, bend left, "a"]
    \end{tikzcd}
  \end{center}
  with the single relation $\rho=\{ba\}$.  That is, the following categories are equivalent 
  \[
   \und{\F}_2\MMod \simeq \CohMack_{\F_2}(C_2) \simeq \Fun((Q, \rho), \Vect_{\F_2}).
  \]
\end{example}

\begin{proof}
  This is a simple change of notation and change of basis. Let $a =\res_e^{C_2}$ and $b =\tr_e^{C_2}$.  Let $c$ denote $\id - \gamma$.  Then an $\und{\F}_2$-module is a functor from the quiver $Q$
  \begin{center}
    \begin{tikzcd}
      \bullet \arrow[loop left, "c", out=210, in=150, looseness=5] \ar[r, bend left, "b"] & \bullet \ar[l, bend left, "a"]
    \end{tikzcd}
  \end{center}
  to $\Vect_{\F_2}$.  The relations from \cref{ex:FpModMackey} imply
  \[
  ca=0, \quad bc=0, \quad c^2=0, \quad ab = c, \quad \text{and} \quad ba = 0.
  \]
  Since $ab = c$, the loop $c$ is not needed.  After making this substitution, the relation $ba = 0$ implies the rest. 
\end{proof}

Using \cref{ex:C2constantquiver} and \cref{prop:IndecompProjectives}, we can characterize indecomposable projective $\und{\F}_2$-modules. 

\begin{example}\label[example]{ex:C2constant_projectives}
  Writing the bound quiver that defines $\und{\F}_2$-modules as
  \begin{center}
    \begin{tikzcd}
      1 \ar[r, bend left, "b"] & 0 \ar[l, bend left, "a"] & ba=0,
    \end{tikzcd}
  \end{center}
  we see the two indecomposable projectives (over an arbitrary field $k$) are  
  \begin{center}$P_1$ = 
    \begin{tikzcd}
      k^2 \ar[r, shift left=1.25, "\begin{psmallmatrix} 1 \ \ 0 \end{psmallmatrix}"] & k\phantom{^2} \ar[l, shift left=1.25, "\begin{psmallmatrix} 0 \\ 1 \rule{0pt}{7pt}\end{psmallmatrix}"]
    \end{tikzcd} \quad \text{and} \quad
$P_0$ = 
    \begin{tikzcd}
      k \ar[r, "0",shift left=1.25] & k \ar[l,shift left=1.25, "1"].
    \end{tikzcd}
  \end{center}
  Over $\F_2$, these coincide with $F=P_1$ (after a change of basis) and $H=P_0$ in the notation of \cite{DHM}.
\end{example}

For larger primes, a similar substitution and reduction applied to \cref{ex:FpModMackey} gives the following.

\begin{example}\label[example]{ex:Cpconstantquiver}
  Let $G=C_p$ for $p$ an odd prime.  An $\und{\F}_p$-module is a quiver representation of the following bound quiver $Q$
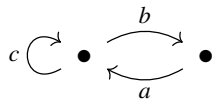
  \begin{center}
    \begin{tikzcd}
      \bullet \arrow[loop left, "c", out=210, in=150, looseness=5] \ar[r, bend left, "b"] & \bullet \ar[l, bend left, "a"]
    \end{tikzcd}
  \end{center}
  with relations $\rho = \{c^p,\ ca,\ bc,\ ba,\ ab - c^{p-1}\}$. Note the relation $c^p$ is redundant.
\end{example}

More generally, we can describe not just the cohomological Mackey functors but every Mackey functor in $\Mack_{\F_p}(C_p)$ (i.e.\ every $\und{A}\otimes{\F_p}$-module) using quivers and the Lewis diagram in \cref{ex:BurnsideMackey}.

\begin{example}
  Let $G=C_p$ for $p$ any prime.  A Mackey functor in $\Mack_{\F_p}(C_p)$ is a quiver representation of the following bound quiver $Q$
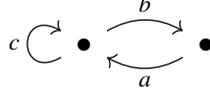
  \begin{center}
    \begin{tikzcd}
      \bullet \arrow[loop left, "c", out=210, in=150, looseness=5] \ar[r, bend left, "b"] & \bullet \ar[l, bend left, "a"]
    \end{tikzcd}
  \end{center}
  with relations $\rho = \{c^p,\ ca,\ bc,\ ab - c^{p-1}\}$.
\end{example}

As an immediate consequence, when $p=2$ we have the following characterization.

\begin{example}\label[example]{ex:C2characteristic2MackeyQuiver}
  Let $G=C_2$.  A Mackey functor in $\Mack_{\F_2}(C_2)$ is a quiver representation of the following bound quiver $Q$
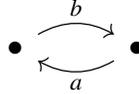
  \begin{center}
    \begin{tikzcd}
      \bullet \ar[r, bend left, "b"] & \bullet \ar[l, bend left, "a"]
    \end{tikzcd}
  \end{center}
  with relations $\rho = \{aba,\ bab\}$.
\end{example}

\begin{remark}
  One can consider representations of the bound quiver 
  \begin{center}
    $\left.
    \begin{tikzcd}
      \bullet \ar[r, bend left, "b"] & \bullet \ar[l, bend left, "a"]
    \end{tikzcd}
        \middle/ 
      \langle ba \rangle
      \right.$
  \end{center}
from \cref{ex:C2constantquiver} over any field $k$.  In fact, over $\C$, this bound quiver appears in the study of $\mathfrak{sl}_2$ representations and category $\mathcal{O}$, and its representations are well-known.  However, it is only in characteristic $2$ that this particular bound quiver bears any relation to cohomological $C_2$-Mackey functors. 
\end{remark}

\subsection{The Mackey and cohomological Mackey algebras} In general, as explained in \cite{TW95}, the category $\Mack_R(G)$ is equivalent to the category of modules of an $R$-algebra $\mu_R(G)$, called the \emph{Mackey algebra}.  This algebra $\mu_R(G)$ is the path algebra of the bound quiver with vertices $\{G/H \mid H \leq G\}$, arrows given by restrictions, transfers, and conjugations, and relations determined by \cref{def:MackeyFunctor}.  Thus, with the notation above, we have the following.

\begin{theorem}
Let $G$ be a finite group and let $R$ be a commutative ring.  There is an equivalence of categories
\[
\Mack_R(G) \simeq \und{A}_R\MMod \simeq \mu_R(G)\MMod.
\]
\end{theorem}

From this equivalence, it is immediate that when $R=k$ is a field, $\mu_k(G)$ is a finite dimensional $k$-algebra satisfying the Krull--Schmidt theorem.  Thus, any finite dimensional $\mu_k(G)$-module decomposes uniquely into a direct sum of indecomposables (up to isomorphism and order of summands).

It is shown in \cite{TW95} (as well as \cite{TWSimple90}) that when the order of the group is invertible, the Mackey algebra is semisimple.

\begin{proposition}(\cite[Theorem~3.5]{TW95})\label[proposition]{prop:invertible_characteristic}
  Let $k$ be a field of characteristic $0$ or characteristic prime to $|G|$.  Then $\mu_k(G)$ is a semisimple $k$-algebra. 
\end{proposition}

Incorporating the additional relations for cohomological Mackey functors, the category $\CohMack_R(G)$ is equivalent to the category of modules of an $R$-algebra $\cohmu_R(G)$, called the \emph{cohomological Mackey algebra}.

Then, by \cref{prop:invertible_characteristic} and the fact that $\CohMack_k(G)$ is a full subcategory of $\Mack_k(G)$, when $|G|$ is invertible the cohomological Mackey algebra is semisimple as well.

\begin{corollary}\label[corollary]{cor:cohmu_invertible_characteristic}
  Let $k$ be a field of characteristic $0$ or characteristic prime to $|G|$.  Then $\cohmu_k(G)$ is a semisimple $k$-algebra.
\end{corollary}

Yoshida identified the cohomological Mackey algebra using permutation modules.

\begin{definition}\label{def:cohomological_mackey_algebra}
  For a finite group $G$ and a commutative ring $R$, the \emph{Hecke category} $\mathcal{H}_R(G)$ is defined as the full subcategory of $RG\MMod$ consisting of the permutation modules $R[G/H]$ for $H \leq G$.  
  
  Yoshida showed \cite[Theorem~4.3]{Yos83} that cohomological Mackey functors are equivalent to representations of the Hecke category
  \begin{align*}
    \CohMack_{R}(G) \simeq \operatorname{Add}_R(\mathcal{H}_R, R\MMod).
  \end{align*}
  Since $\mathcal{H}_R$ is a finite category, the category of representations is equivalent to the module category $\hecke_R(G)\MMod$, where the Hecke algebra is the endomorphism ring 
  \[
  \hecke_R(G) \coloneqq \End_{RG}\left(\bigoplus_{H \leq G} R[G/H]\right).
  \] 
  Moreover, as observed in \cite{TW95}, $\cohmu_R(G) = \hecke_R(G)^{\op}$. Since $\Hom_R(-, R)$ gives a duality on the Hecke category we also have $\hecke_R(G)^{\op} \cong \hecke_R(G)$, so we may work with $\hecke_R(G)$ instead of $\cohmu_R(G)$ when  convenient.
  
  With the notation above, we have the following theorem due to Yoshida \cite{Yos83} and described in \cite{TW95}.
\end{definition}
  \begin{theorem} Let $G$ be a finite group and $R$ a commutative ring.  There is an equivalence of categories
  \[
   \und{R}\MMod \simeq  \CohMack_{R}(G) \simeq \cohmu_R(G)\MMod \simeq \hecke_R(G)^{\op}\MMod.
  \]
  \end{theorem}

  Skipping ahead somewhat, one of our main examples throughout will be the cohomological Mackey algebra of $C_p$ over a field $k$ of characteristic $p$, as deduced from Yoshida's characterization.  See \cref{prop:Cp^mCohMackQuiver} and \cref{prop:CpCohMackQuiver} for the details.

  \begin{proposition}
Let $k$ be a field of characteristic $p$.  The cohomological Mackey algebra $\cohmu_k(C_p)$ is the path algebra of the quiver $Q$
  \begin{center}
    \begin{tikzcd}
      1 \arrow[loop left, "c", out=210, in=150, looseness=5] \ar[r, bend left, "b"] & 0 \ar[l, bend left, "a"]
    \end{tikzcd}
  \end{center}
  with relations
  \[
  \rho = \{c^p,\ ba,\ ab - c^{p-1},\ ca,\ bc\},
  \]
  as in \cref{ex:Cpconstantquiver} for $k = \F_p$.
  
  In the case $p=2$, the arrow $c=ab$ is redundant and the cohomological Mackey algebra 
  $\cohmu_k(C_2)$ is the path algebra of the quiver $Q$
  \begin{center}
    \begin{tikzcd}
      1 \ar[r, bend left, "b"] & 0 \ar[l, bend left, "a"]
    \end{tikzcd}
  \end{center}
  with relation
  \[
  \rho = \{ba\},
  \]
  recovering \cref{ex:C2constantquiver} for $k = \F_2$.

\end{proposition}

\subsection{Equivariant cohomology and Eilenberg--MacLane spectra}\label{sec:Eilenberg-Maclane-spectra}
Mackey functors are fundamental to the study of equivariant cohomology of $G$-spaces and to equivariant stable homotopy theory.  They often play the role of abelian groups in equivariant homotopy theory.  For example, in the equivariant setting, stable homotopy groups are replaced by stable Mackey functors.  Mackey functors also appear in equivariant cohomology.

Bredon defined a $\Z$-graded equivariant cohomology theory for $G$-CW complexes that takes coefficients in a coefficient system \cite{Bredon}.  In order to have Bredon cohomology of $G$-spaces extend to an $RO(G)$-graded cohomology theory (graded by real representations) and be represented by a genuine $G$-spectrum, the coefficient system must extend to a Mackey functor.  
The representing object of an ordinary equivariant cohomology theory with coefficients in a Mackey functor $M$ is a genuine $G$-spectrum called an Eilenberg--MacLane spectrum.  For more details and background, see for example \cite{Mayetal96}, \cite{Shu10}, or \cite{HillHHT}.

An equivariant Eilenberg--MacLane spectrum is characterized by its $\Z$-graded homotopy as follows.

\begin{definition}
  Let $M$ be a $G$-Mackey functor.  The Eilenberg--MacLane spectrum $HM$ is the genuine $G$-equivariant spectrum defined (uniquely up to homotopy) by the property that the $\Z$-graded homotopy Mackey functors are
  \[
  \und{\pi}_n^G HM = \begin{cases}
    M & n = 0 \\
    0 & \text{otherwise.}
  \end{cases}
  \] 
\end{definition}

In the case that $M$ is a Mackey ring, $HM$ is a ring spectrum.  If $M$ is a commutative Mackey ring, $HM$ can be taken to be a commutative ring spectrum (see \cite{Ullman}).  In fact, both the Burnside Mackey functor and the constant Mackey functor have the extra structure of a Tambara functor, and whenever $M$ is a Tambara functor, $HM$ is a genuine $G$-commutative ring spectrum (see \cite{Ullman} and \cite{BH15}).

The following result is due to \cite{SS03}, as explained in \cite{Zeng} for $M = \und{\Z}$, and described more generally in \cite{GS14}.

\begin{theorem}[\cite{SS03}]\label{Thm:QuillenEquiv}
  Let $M$ be a commutative Mackey ring.  There is a Quillen equivalence
  \begin{center}
\begin{tikzcd}
  \Ch(M\MMod) \arrow[rr,bend left,"\Gamma", out=15, in=165] & \simeq & HM\MMod \arrow[ll, bend left,"\Psi",out=15,in=165]
\end{tikzcd}
\end{center}
between the algebraic category of chain complexes of $M$-modules with projective model structure and the topological category of $HM$-modules.
\end{theorem}

A Quillen equivalence is a relationship between two model categories.  In particular, it implies the homotopy categories are equivalent.  In the context of \cref{Thm:QuillenEquiv}, we have
\[
\mathcal{D}(M) \simeq \Ho(HM\MMod).
\]
Thus to study the homotopy category of $HM$-modules in $G$-spectra, one could study the derived category of the Mackey ring $M$.  Restricting to compact objects, 
\[
\mathcal{D}^{\perf}(M) \simeq \Ho(HM\MMod)^\omega.
\]
This is our motivation to study the perfect complexes $\mathcal{D}^{\perf}(M)$, following the work in \cite{DHM} for $G=C_2$.

\begin{remark}
  It was recently shown by Fuhrmann \cite[Theorem~3.12]{Fuh25} that there is an infinity categorical version of \cref{Thm:QuillenEquiv}.  That is, for any Green functor $M$ there is an equivalence of symmetric monoidal $\infty$-categories 
  \[
  D(M) \simeq HM\MMod.
  \]
\end{remark}

\bigskip
Before we turn to the derived case, we consider the representation type of the cohomological Mackey algebra.

\section{Representation type}\label[section]{sec:reptype}
In this section, we look at the representation type of algebras.  Much of the content here is well-known, but we include a fair amount of detail in order to pay particular attention to algebras over arbitrary fields.  We then review representation type in the special cases of special biserial, string, and gentle algebras.

\subsection{Representation type of algebras} One of the classical questions in representation theory is when an algebra has only finitely many indecomposable representations up to isomorphism. When this is the case, we say that the algebra has finite representation type, and one can hope to list all the isomorphism classes of indecomposables. 

Over an algebraically closed field $k = \bar{k}$, one can divide the class of algebras of infinite representation type into algebras of tame representation type and those of wild representation type.  It is a celebrated theorem of Drozd that any finite dimensional algebra falls into exactly one of the three classes: finite, tame, or wild type.  See for example \cite{Dro77,CB88,BSZ09}.

Intuitively, tame algebras are those whose infinitely many isomorphism classes of indecomposable representations can be completely ``described'' and they are parametrized by a single parameter. On the other hand, wild algebras have hopelessly complicated representation theory with arbitrarily large families of indecomposable representations.  These families are so large, each wild $k$-algebra contains the representation theory of every finite dimensional $k$-algebra.  It is therefore rather surprising that an algebra with infinitely many indecomposables must be either tame or wild.

It is common to collect algebras with ``describable'' representation theory, namely those of finite and tame representation type, and refer to both as tame.  With this terminology, Drozd's famous theorem is often referred to as the tame-wild dichotomy. 

The term describable is meant to indicate a complete and meaningful classification, not just a bijection with another collection of objects.  As an example, although $\Z$ is not a field, 
it is instructive to think of the fundamental theorem of finitely-generated $\Z$-modules.  Here there are infinitely many indecomposables, but we can completely classify them; indecomposables are of the form $\Z$ or $\Z/(p^n\Z)$ for any choice of prime $p$ and positive integer $n$.

In the literature, the three representation types (finite, tame, and wild) are typically defined for $k$-algebras over an algebraically closed field $k=\bar{k}$.  To talk about these notions over arbitrary fields, one usually extends scalars to the algebraic closure (or at least separable closure).  We are interested in finite fields, so we naively extend the definitions of finite and wild representation types to arbitrary fields.  It is less clear how to meaningfully extend tame representation type to arbitrary fields and we do not pursue it.

\begin{definition}[Finite and infinite representation type]
A $k$-algebra is said to be of \emph{finite representation type} or simply to be \emph{representation finite} if it has only finitely many isomorphism classes of indecomposable modules. Otherwise, the algebra is said to have \emph{infinite representation type} or be \emph{representation infinite}.
\end{definition}

\begin{example}
  The group algebra $kG$ for $k$ a field of characteristic $p$ and $G=C_{p^m}$ is representation finite.  Presenting $kG$ as $k[x]/\left(x^{p^m}\right)$, the indecomposables are of the form $k[x]/(x^r)$ for $1 \leq r \leq p^m$. 
\end{example}

\begin{example}
  For $k$ a field of characteristic $p$ and $G = C_{p} \times C_{p}$, the group algebra $k[C_{p} \times C_{p}]$ is representation infinite.  See \cite{Hig54}.
\end{example}

 As in the example of the structure theorem for modules over a PID, the indecomposables of a $k$-algebra $\Lambda$ of tame representation type should have a nice description. The definition of tame below leverages the structure theorem by reducing the description of $\Lambda$-modules to the description of the simple $k[x]$-modules, $k[x]/(x-\lambda)$ for $\lambda \in k$. The remarkable thing about this definition is that it turns out to be exactly the complement of wild (\cref{def:wild}) over an algebraically closed field.

\begin{definition}[Tame representation type]
Over an algebraically closed field $k=\bar{k}$, an algebra $\Lambda$ has \emph{tame representation type} or is called \emph{tame} if for each dimension $d>0$ there is a finite number of $\Lambda$--$k[x]$-bimodules $M_i$ that are free as right $k[x]$-modules, and such that every $d$-dimensional indecomposable $\Lambda$-module is isomorphic to $M_i \otimes_{k[x]} N$ for some $M_i$ and some simple $k[x]$-module $N$.
\end{definition}

\begin{example}
  The canonical example of a tame algebra is $\Lambda = k[x]$ for $k = \bar{k}$ algebraically closed. It has $d$-dimensional indecomposables of the form $k[x]/(x-\lambda)^d$, i.e.\ Jordan blocks parametrized by eigenvalue $\lambda$.
\end{example}

While the polynomial algebra $k[x]$ is a PID for any field $k$, the simples are of varying dimension. Thus one would need to replace the notion of dimension in the definition of tame to use the polynomial algebra to parametrize indecomposables. Such generalizations have been made, however no dichotomy theorem is known in these contexts. See \cite{CB91} for a notion of tame for arbitrary rings.

\begin{example}
 Let $k = \bar{k}$ be an algebraically closed field of characteristic 2.  The group algebra $k[C_2 \times C_2]$ has infinite but tame representation type.  A description of the indecomposables can be found in \cite[Section~11.5]{WebbCourseNotes}.
\end{example}

In order to define wild representation type, we need the notion of a representation embedding.  We return to arbitrary fields. 

\begin{definition}[Representation embedding]\label[definition]{def:rep_emb}
  An additive functor $F\colon \mathcal C \to \mathcal D$ between abelian categories is called a \emph{representation embedding} if the functor $F$
  \begin{enumerate}
    \item is exact,
    \item is essentially injective, i.e.\ for any $M,N$ in $\mathcal C$, $FM \cong FN$ implies $M\cong N$, and
    \item preserves indecomposability, i.e.\ $FM$ is indecomposable whenever $M$ is. 
  \end{enumerate}
\end{definition}

\begin{definition}[Wild representation type]\label[definition]{def:wild}
  A $k$-algebra $\Lambda$ is said to have \emph{wild representation type}, or simply to be \emph{wild}, if for every finite dimensional $k$-algebra $\Gamma$ there is a representation embedding $\Gamma\mmod \to \Lambda\mmod$.
\end{definition}
\begin{remark}
  Recall the notation $\Lambda\mmod$ means the category of finite dimensional $\Lambda$-modules, but we do not require $\Lambda$ to be finite dimensional as an algebra.
\end{remark}

By definition, a classification of all indecomposables of a wild algebra would require a classification of the indecomposables of every finite dimensional $k$-algebra.  The canonical example of a wild algebra is the free algebra $k\langle x, y \rangle$ from \cref{ex:freealgebraxy}, as we demonstrate below in \cref{thm:wildfreealgebra}.

Over algebraically closed fields, all finite dimensional algebras fall into one of these types.  The following theorem is due to Drozd \cite{Dro77} (see also \cite[Corollary~C]{CB88} for a proof using the theory of bocses).

\begin{theorem}(Tame-wild dichotomy \cite{Dro77})\label[theorem]{thm:dichotomy}
  Over an algebraically closed field $k = \bar{k}$, a finite dimensional $k$-algebra with infinite representation type is either tame or wild, and not both.
\end{theorem}

Before we turn to the example $k \langle x,y \rangle$, it is useful to have some alternative characterizations of representation embeddings.

\begin{proposition}\label[proposition]{prop:embedding_is_tensorproduct}
  If $\Gamma$ is a finite dimensional algebra and $F\colon \Gamma\mmod \to \Lambda\mmod$ is a representation embedding, then $F(\Gamma)$ has a canonical structure of a $\Lambda$--$\Gamma$-bimodule, where the action of $\Gamma$ is induced by the isomorphism $\End_\Gamma(\Gamma) \cong \Gamma^{\op}$. Moreover, there is a natural isomorphism $F \cong F(\Gamma)\otimes_\Gamma -$ and $F(\Gamma)$ is projective as a $\Gamma$-module.
\end{proposition}

The following proposition demonstrates our key technique for showing that an algebra is wild. 
\begin{proposition}\label[proposition]{prop:key_technique}
  Let $\mathcal{C}$ and $\mathcal{D}$ be idempotent complete additive categories and let $F\colon \mathcal C \to \mathcal D$ be an additive functor. Suppose there exists a collection of (possibly unnatural) additive maps $\Delta_{M,N}\colon \Hom_{\mathcal D}(FM, FN) \to \Hom_{\mathcal C}(M, N)$ 
  such that the following two conditions hold:
  \begin{enumerate}
    \item a map $f\colon F(M) \to F(N)$ is an isomorphism if and only $\Delta_{M,N}(f)\colon M \to N$ is, and
    \item whenever $e\colon F(M) \to F(M)$ is idempotent, then $\Delta_{M,M}(e)$ is as well.
  \end{enumerate}
  Then $F$ is essentially injective and preserves indecomposability.
    \begin{center}
    \begin{tikzcd}
      \Hom_{\mathcal C}(M, N) \ar[r, "F", bend left, out=20, in=160] & \Hom_{\mathcal D}(FM, FN) \ar[l, bend left, "\Delta_{M,N}", dashed, out=20, in=160]
    \end{tikzcd}
  \end{center}
  
  In particular, if $F$ is an exact functor between abelian categories, these properties are sufficient to show that $F$ is a representation embedding.
  \begin{proof}
    Since an isomorphism $F(M) \cong F(N)$ implies an isomorphism $M \cong N$, the functor $F$ is essentially injective. To see that $F$ preserves indecomposability, assume $F(M)$ is decomposable. Then there exists a nontrivial idempotent $e\colon F(M) \to F(M)$ given by projection onto a nontrivial summand. This means $e$ and $1_{F(M)} - e$ are both idempotent, neither are isomorphisms, and they sum to the identity.  Now applying $\Delta$ ($=\Delta_{M,M}$) we see that $\Delta(e)$ and $\Delta(1_{F(M)}-e)$ are both idempotent. Since they cannot be isomorphisms, they are not the identity.  On the other hand, their sum $\Delta(e) + \Delta(1_{F(M)}-e) = \Delta(1_{F(M)})$ is an isomorphism, so they also cannot be $0$. Thus $\Delta(e) \colon M \to M$ is a nontrivial idempotent and $M$ is decomposable.
  \end{proof}
\end{proposition}

In particular, fully faithful exact functors are representation embeddings.

\begin{corollary}\label[corollary]{cor:fullyfaithful_is_rep_embedding}
  If a functor $F\colon \mathcal C \to \mathcal D$ between abelian categories is fully faithful and exact, then $F$ is a representation embedding.
\end{corollary}

\begin{proof}
  Take $\Delta_{M,N}$ to be the inverse of the bijection $F_{M,N}$ induced by $F$.
\end{proof}

\begin{example}
   Let $f \colon R \to S$ be a surjective map of $k$-algebras.  Then restriction of scalars is a representation embedding $S\MMod \to R\MMod$.
\end{example}

We now use \cref{prop:key_technique} to show there is a representation embedding between free algebras of rank greater than 1.  Because it is common in the literature to work over an algebraically closed field, we include the proof of this and other well-known results in order to illustrate the techniques for an arbitrary field.

\begin{proposition}\label[proposition]{prop:free_algebra_embeds_in_Sigma}
  Let $k$ be an arbitrary field, and let $\Gamma = k\langle x_1, \cdots, x_n \rangle$ for $n >1$ and $\Sigma = k\langle a, b \rangle$ be free algebras. There is a representation embedding 
  \[
  \Gamma\mmod \to \Sigma\mmod.
  \]
  
  \begin{proof}
    We will define a functor $F\colon \Gamma\mmod \to \Sigma\mmod$ and show that it is a representation embedding using \cref{prop:key_technique}.

    Let $V$ be $\Gamma$-module of dimension $m$ and choose a basis for $V$. Let $X_1, \dots X_n$ be $m \times m$ matrices representing that action of $x_1, \dots, x_n$ on $V$.   Then for $(V, X_1, \cdots, X_n)$ in $\Gamma\mmod$ we define $F(V) \coloneqq (V^n, A, B)$, with $A$ and $B$ block matrices given by 
    \begin{align*}
      A &= \begin{pmatrix}
        X_1 & \\
        1 & X_2 \\
         & 1 & X_3\\
        & &\ddots & \ddots\\
        & & & 1 & X_n
      \end{pmatrix}, 
      \qquad 
      B = \begin{pmatrix}
        1 \\
        & 0 \\
        & & 1\\
        & & & 0\\
        &&&&\ddots
      \end{pmatrix},
    \end{align*}
    where 1 denotes the identity on $V$.

    Given $(V, X_1, \cdots, X_n)$ and $(U, Y_1, \cdots, Y_n)$, as well as a homomorphism $f\colon V \to U$ in $\Gamma\mmod$, we define $F(f)$ as $f^{\oplus n} \colon V^{n} \to U^{n}$. One readily checks that this is a homomorphism of $\Sigma$-modules and that $F$ is an exact functor.

    To employ the technique of \cref{prop:key_technique}, we need to take a homomorphism $F(V) \to F(U)$ and construct a homomorphism $V \to U$. 
    Let $\Phi\colon F(V) \to F(U)$ be any homomorphism of $\Sigma$-modules given as a block matrix with blocks $\Phi_{ij}$ of size $\dim U \times \dim V$. Commutativity $\Phi B_V = B_U \Phi$ implies that $\Phi_{ij}=0$ whenever $i$ and $j$ have different parity.
    
    Computing $\Phi A_V$ and $A_U\Phi$ we get 
    \begin{align*}
      (\Phi A_V)_{ij} &= \Phi_{ij}X_j + \Phi_{i,j+1}\\
      (A_U\Phi)_{ij} &= \Phi_{i-1,j} + Y_i\Phi_{ij}
    \end{align*}
    where $\Phi_{ij} = 0$ when the indices are out of bounds.

    When $i$ and $j$ have different parity $\Phi_{ij} = 0$, so the equation $\Phi A_V = A_U\Phi$ tells us that $\Phi_{i-1,j} = \Phi_{i,j+1}$. In particular this means that the top row $\Phi_{1j}$ is $0$ except for $\Phi_{11}$.
    
    When $i$ and $j$ have the same parity, $\Phi_{i-1, j}$ and $\Phi_{i,j+1}$ are both $0$, so $\Phi A_V = A_U\Phi$ implies $\Phi_{ij}X_j = Y_i\Phi_{ij}$. In particular, we have $\Phi_{ii}X_i = Y_i\Phi_{ii}$.

    The above computations imply that $\Phi$ is block-lower triangular with $\Phi_{ii}=\Phi_{j\!j}$ for all $i$ and $j$, and that $\Phi_{ii}X_i = Y_i\Phi_{ii}$. This last fact means that $\Phi_{11} \colon V \to U$ is a homomorphism of $\Gamma$-modules. We therefore define $\Delta(\Phi) = \Phi_{11}$ in accordance with \cref{prop:key_technique}. Since $\Phi$ is lower triangular, it is an isomorphism if and only if $\Phi_{ii} = \Phi_{11}$ is. Further if $V=U$ and $\Phi$ is idempotent, then $\Phi_{11}$ is as well. So the criteria of \cref{prop:key_technique} are satisfied and $F$ is a representation embedding.
  \end{proof}
\end{proposition}

\begin{theorem}\label[theorem]{thm:wildfreealgebra}
  The free algebra on two generators $\Sigma = k\langle a, b\rangle$ is wild.
  \begin{proof}
    Let $\Gamma$ be a any finite dimensional algebra. Then there exists a surjective homomorphism $k\langle x_1, \cdots, x_n \rangle \to \Gamma$ for some $n >1$. The restriction of scalars $\Gamma\mmod \to k\langle x_1, \cdots, x_n \rangle\mmod$ is fully faithful and exact, and thus a representation embedding by \cref{cor:fullyfaithful_is_rep_embedding}. Composing this with the embedding in \cref{prop:free_algebra_embeds_in_Sigma} we get a representation embedding of $\Gamma\mmod$ into $\Sigma\mmod$, hence $\Sigma$ is wild.
  \end{proof}
\end{theorem}

We can compose representation embeddings, so the following is immediate.
  
\begin{proposition}\label[proposition]{prop:embedding_from_wild_implies_wild}
  If $\Gamma$ is wild, and $F\colon \Gamma\mmod \to \Lambda\mmod$ is a representation embedding, then $\Lambda$ is wild as well.
\end{proposition}

A useful technique to show an algebra is wild is to find a representation embedding from $k\langle x, y\rangle\mmod$ or another wild algebra.  We illustrate this with an example of a wild path algebra that will be of importance to us in the proof of \cref{prop:wild_cyclic_group}.

\begin{proposition}\label[proposition]{prop:A1tildetilde_wild}
  Let $\Lambda$ be the path algebra of the quiver below (with no relations). Then $\Lambda$ is wild. 
  \begin{center}
    \begin{tikzcd}
      1 \ar[r, bend left, "a"] \ar[r, bend right, swap, "b"] & 2 \ar[r, "w"] & 3
    \end{tikzcd}
  \end{center}
  \begin{proof}
    Let $\Sigma = k\langle x, y \rangle$ be the free algebra on two generators, and consider the $\Lambda$--$\Sigma$-bimodule $P$ given by 
    \begin{center}
      \begin{tikzcd}[ampersand replacement = \&]
        \Sigma^2 \ar[r, bend left]{}{\begin{pmatrix} x & y \\ 1 & 0 \end{pmatrix}} \ar[r, bend right, swap]{}{\begin{pmatrix} 1 & 0 \\ 0 & 1 \end{pmatrix}}  \& \Sigma^2 \ar[r]{}{\begin{pmatrix} 0 & 1 \end{pmatrix}}  \& \Sigma
      \end{tikzcd}
    \end{center}
    We see that this is free as a $\Sigma$-module, so $P\otimes_\Sigma -$ is exact.

    For a $\Sigma$-module $(V, X, Y)$ the tensor product $P\otimes_{\Sigma} V$ is given by
    \begin{center}
      \begin{tikzcd}[ampersand replacement = \&]
        V^2 \ar[r, bend left]{}{\begin{pmatrix} X & Y \\ 1 & 0 \end{pmatrix}} \ar[r, bend right, swap]{}{\begin{pmatrix} 1 & 0 \\ 0 & 1 \end{pmatrix}}  \& V^2 \ar[r]{}{\begin{pmatrix} 0 & 1 \end{pmatrix}}  \& V
      \end{tikzcd}
    \end{center}
    A straightforward computation shows that $P\otimes_\Sigma -$ is fully faithful. Hence by \cref{cor:fullyfaithful_is_rep_embedding}, $P\otimes_\Sigma -$ defines a representation embedding $\Sigma\mmod \to \Lambda\mmod$ and $\Lambda$ is wild.
  \end{proof}
\end{proposition}

\subsection{Special biserial, string, and gentle algebras}
There are many notable types of algebras with well-known representation theory.  Here we look at special biserial algebras, string algebras, and gentle algebras.

What makes special biserial algebras `special' is that they are tame, and their representation theory can be described in terms of so-called ``strings'' and ``bands''. Loosely speaking, a string is an unoriented path of the quiver that does not pass through a relation.  A band is a minimal cyclic string such that all powers of it are also strings.

\begin{definition}[Special biserial algebra]\label[definition]{def:special_biserial}
  A bound quiver $(Q,\rho)$ (or its admissible path algebra) is called \emph{special biserial} if the following conditions hold. 
  \begin{enumerate}
    \item At each vertex there are at most two outgoing and at most two incoming arrows.
    \item For each arrow $\alpha$ there exists at most one composable arrow $\beta$ such that $\alpha \beta$ \textbf{is not} a relation, and at most one composable arrow $\gamma$ such that $\gamma \alpha$ \textbf{is not} a relation.
    \item The relations $\rho$ are admissible, i.e.\ they consist of linear combinations of paths of length at least $2$, and there exists an $n$ such that all paths of length more than $n$ are contained in the ideal generated by the relations. 
  \end{enumerate}
\end{definition}

\begin{remark}
  Note the conditions for a special biserial algebra imply we may assume the relations are only monomial (given by a single path, also called a zero-relation) or binomial (a linear combination of two paths, also called a commutativity relation). 
\end{remark}

\begin{definition}[String algebra]\label[definition]{def:string_algebra}
 A special biserial bound quiver $(Q,\rho)$ (or its admissible path algebra) is \emph{string} if the conditions in \cref{def:special_biserial} hold and additionally the following:
 \begin{enumerate}
 \item[(4)] $\rho$ consists of zero-relations, i.e.\ paths of length at least $2$ (and not linear combinations).
 \end{enumerate}
\end{definition}

Imposing even stronger conditions, we have the notion of a gentle algebra.

\begin{definition}[Gentle algebra]\label[definiton]{def:gentle}
  A string bound quiver $(Q,\rho)$ (or its admissible path algebra) is \emph{gentle} if the conditions in \cref{def:special_biserial} hold and additionally the following hold.
  \begin{enumerate}
    \item[(5)] All relations in $\rho$ have length exactly $2$.
    \item[(6)] For each arrow $\alpha$ there exists at most one composable arrow $\beta'$ such that $\alpha \beta'$ \textbf{is} a relation, and at most one composable arrow $\gamma\,'$ such that $\gamma\,' \alpha$ \textbf{is} a relation.
  \end{enumerate}
\end{definition}

\begin{example}\label[example]{ex:Cp_generalization_special_biserial}
    The bound quiver below
  \begin{center}
    $\left.
    \begin{tikzcd}
      1 \arrow[loop left, "c", out=210, in=150, looseness=5] \ar[r, bend left, "b"] & 0 \ar[l, bend left, "a"]
    \end{tikzcd}
    \middle/ 
    \langle c^n,\ ba,\ ab-c^{n-1},\ ca,\ bc \rangle
    \right.$,
  \end{center}
  generalizing $\cohmu_k(C_p)$ from \cref{prop:CpCohMackQuiver}, is special biserial.  
  It is not string because the linear combination $ab-c^{n-1}$ is not a path.  Even if this relation were deleted, the result would not be gentle because both $ba$ and $ca$ are relations. 
\end{example}

\begin{example}\label[example]{ex:C2_constant_gentle}
  When $k$ has characteristic $2$, the algebra $\cohmu_k(C_2)$ is gentle.
  \begin{center}
   $\left.
    \begin{tikzcd}
      1 \ar[r, bend left, "b"] & 0 \ar[l, bend left, "a"]
    \end{tikzcd}
    \middle/ 
    \langle ba \rangle.
    \right.$
  \end{center}
  See \cref{ex:C2constantquiver} and \cref{prop:CpCohMackQuiver}. 
\end{example}

The indecomposable representations of special biserial algebras were classified in \cite[Proposition~2.3]{WW85}. Note that the authors in \cite{WW85} assume the field is algebraically closed, but the argument generalizes to arbitrary fields.

Given a special biserial algebra, one can construct a quotient that is a string algebra with a nice relationship; all but finitely many of the indecomposables of the special biserial algebra are indecomposables of the associated string algebra.  
One obtains the string algebra by replacing each binomial relation with a pair of monomial relations, for example replacing the relation $ab-c^{n-1}$ from \cref{ex:Cp_generalization_special_biserial}  with $ab,\, c^{n-1}$. 
The indecomposables of the special biserial algebra that are not also modules of the associated string algebra are exactly the indecomposable biserial projective-injectives. For more details, see \cite[Theorem~II.1.3]{Erd90}. 

Butler and Ringel \cite{BR87} showed (over arbitrary fields) the indecomposable modules of a finite dimensional string algebra are divided into two types called string modules and band modules, and the algebra is representation finite if and only if there are no band modules. 
We outline the key points in describing strings and then their associated modules.  
See \cite{BR87} for more details or \cite{CB18} for a generalization to infinite dimensional string algebras.

Let $(Q,\rho)$ be the bound quiver of a string algebra.  
For an arrow $\alpha$, let $\alpha^{-1}$ or $\alpha^-$ denote the formal inverse of $\alpha$ with $s(\alpha^{-1})= t(\alpha)$ and $t(\alpha^{-1})$ = $s(\alpha)$.  
A \emph{string} $\omega$ is a walk through composable arrows and inverse arrows that is reduced (contains no subwalks of the form $\alpha \alpha^{-1}$ or $\alpha^{-1} \alpha$) and avoids relations (contains no subwalks $w \in \rho$ or $w^{-1} \in \rho$).  
One allows for trivial walks, remaining at a vertex. 
A string is called \emph{direct} if it consists only of  arrows in $Q$, \emph{inverse} if it consists only of inverse arrows, and \emph{mixed} otherwise.  
Trivial strings are considered both direct and inverse.

\begin{example}\label[example]{ex:C2constant_strings}
  Consider the bound quiver
   \begin{center}
   $\left.
    \begin{tikzcd}
      1 \ar[r, bend left, "b"] & 0 \ar[l, bend left, "a"]
    \end{tikzcd}
    \middle/ 
    \langle ba \rangle
    \right.$
  \end{center}
  describing $\cohmu_k(C_2)$ in characteristic $2$.  This is a (gentle) string algebra with direct strings
  \[
  e_0,\ e_1,\ a,\ b,\ ab
  \]
  and inverse strings
  \[
  e_0,\ e_1,\ a^{-},\ b^{-},\ b^{-}a^{-}.
  \]
  In this small example, there are no mixed strings.
\end{example}

\begin{example}\label[example]{ex:Cp_quotient_string_algebra}
   Consider the special biserial algebra $\Lambda$ described by the bound quiver from \cref{ex:Cp_generalization_special_biserial}, as shown below with minimal relations.
     \begin{center}
    $\left.
    \begin{tikzcd}
      1 \arrow[loop left, "c", out=210, in=150, looseness=5] \ar[r, bend left, "b"] & 0 \ar[l, bend left, "a"]
    \end{tikzcd}
    \middle/ 
    \langle ba,\ ab-c^{n-1},\ ca,\ bc \rangle
    \right.$
  \end{center}
     This special biserial algebra has associated string algebra $\bar{\Lambda}$
        \begin{center}
    $\left.
    \begin{tikzcd}
      1 \arrow[loop left, "c", out=210, in=150, looseness=5] \ar[r, bend left, "b"] & 0 \ar[l, bend left, "a"]
    \end{tikzcd}
    \middle/ 
    \langle ba,\ ab,\ c^{n-1},\ ca,\ bc \rangle
    \right.$
  \end{center}
  with direct strings
    \begin{align*}
      e_0,\ e_1,\ a,\ b,\ c^i,
    \end{align*}
  inverse strings
  \begin{align*}
      e_0,\ e_1,\ a^-,\ b^-,\ c^{-i},
    \end{align*}
  and mixed strings
  \begin{align*}
      a^-c^i,\ c^{-i}a,\ c^i b^-,\ bc^{-i},\ a^-c^i b^-,\ bc^{-i}a,
    \end{align*}
   where $1 \leq i \leq n-2$. 
\end{example}

From a string $\omega$ associated to a string algebra, a string module $M(\omega)$ is formed by assigning a copy of $k$ each time the string passes a vertex on the walk.  The maps send each copy of $k$ to the next via the identity whenever the walk passes a (possibly inverse) arrow, as in the example below.
\begin{example}\label[example]{ex:string_module}
  The string $bc^{-2}a$ from the string algebra $\bar{\Lambda}$ in \cref{ex:Cp_quotient_string_algebra} has string module depicted 
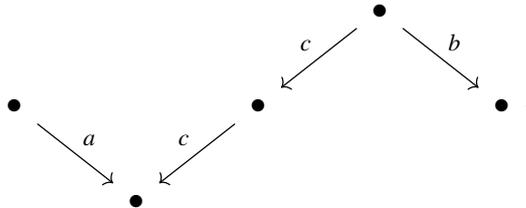
\begin{center}
  \begin{tikzcd}
    & &&   1 \ar[dr, "b"] \ar[dl, swap, "c"] & \\
    0 \ar[dr,"a"] &  &  1 \ar[dl,swap,"c"]  && 0 \\
     & 1  & & 
  \end{tikzcd}.
\end{center}
This is shorthand for the module formed by collecting copies of $k$ labelled by $\{e_0, a, c^{-1}a, c^{-2}a, bc^{-2}a\}$, where  multiplication by an arrow acts on the string accordingly. 
Thus the string module associated to $bc^{-2}a$ is of the form
\begin{center}
  \begin{tikzcd}[ampersand replacement=\&]
     k_{a} \oplus k_{c^{-1}a} \oplus k_{c^{-2}a} \arrow[start anchor = south west, end anchor = north west,loop left, out=210, in=150, looseness=5,shift right = 2pt]{}{
      \begin{pmatrix}
        0 & 1 & 0\\
        0 & 0 & 1\\
        0 & 0 & 0
      \end{pmatrix}
    } 
    \ar[r, start anchor={[xshift=-10pt, yshift=0pt]north east}, end anchor={[xshift=0pt, yshift=0pt]north west}, bend left]{}{
      \begin{pmatrix}
        0 & 0 & 0\\
        0 & 0 & 1
      \end{pmatrix}
    }
    \& k_{e_0} \oplus k_{bc^{-2}a}
    \ar[l, start anchor={[xshift=0pt, yshift=0pt]south west}, end anchor={[xshift=-10pt, yshift=0pt]south east}, bend left]{}{
      \begin{pmatrix}
        1 & 0\\
        0 & 0\\
        0 & 0
      \end{pmatrix}
    }
  \end{tikzcd}. 
\end{center}
\end{example}

The string module $M(\omega) \cong M(\nu)$ if and only if $\omega = \nu$ or $\omega = \nu^{-1}$.  A string $\nu$ that is cyclic is a \emph{band} if every power $\nu^n$ is a string for $n \in \mathbb{N}$, but $\nu$ is not a nontrivial power of any string. In the case of a band, the associated module is called a band module and some more complicated notions of equivalence are required.  See for example \cite{WW85, BR87, CB18, BCSGentleModel21, Lak16} for more precise details. 
Using strings and bands, the representation theory of special biserial algebras is completely describable.

\begin{theorem}
  All string and band modules are indecomposable.  Moreover, every indecomposable module over a string algebra is either a string module or a band module.  
  
  Every indecomposable module over a special biserial algebra is a string module, a band module, or a biserial projective-injective module.
\end{theorem}

The classifications of string and band modules are algorithmic.  An applet made by Jan Geuenich for describing the modules (and more) of special biserial algebras, string algebras, and gentle algebras can be found at \href{https://www.math.uni-bielefeld.de/~jgeuenich/string-applet/}{https://www.math.uni-bielefeld.de/~jgeuenich/string-applet/}.

A straightforward calculation gives the indecomposables in the following example.

\begin{proposition}\label[proposition]{prop:counting_special_biserial_strings}
   For $n >2$ the special biserial algebra $\Lambda$ from \cref{ex:Cp_generalization_special_biserial} associated to the bound quiver
     \begin{center}
    $\left.
    \begin{tikzcd}
      1 \arrow[loop left, "c", out=210, in=150, looseness=5] \ar[r, bend left, "b"] & 0 \ar[l, bend left, "a"]
    \end{tikzcd}
    \middle/ 
    \langle ba,\ ab-c^{n-1},\ ca,\ bc \rangle
    \right.$
  \end{center}
   has string modules associated to the strings
    \begin{align*}
      e_0,\ e_1,\ a,\ b,\ c^i,\ a^-c^i,\ c^i b^-,\ a^-c^i b^-,
    \end{align*}
    where $1 \leq i \leq n-2$, for a total of $4(n-1)$ non-isomorphic string modules.  There are no bands.  There is one biserial projective-injective $P_1$.  Thus $\Lambda$ is representation finite with $5 + 4(n-2)$ indecomposables.
\end{proposition}

\section{Representation type of cohomological Mackey functors}\label[section]{sec:rep_type_cohom_mackey}
In this section, we examine the representation type of the cohomological Mackey algebra. It is immediate from \cref{prop:invertible_characteristic} and \cref{cor:cohmu_invertible_characteristic} that both $\mu_k(G)$ and $\cohmu_k(G)$ are representation finite if $|G|$ is invertible in $k$.  The algebras are semisimple so every module is a direct sum of simple modules.

\begin{proposition}
Let $k$ be a field in which $|G|$ is invertible.  Then $\mu_k(G)$ and $\cohmu_k(G)$ are representation finite.
\end{proposition}

Thus we turn our attention to prime characteristic dividing the order of $G$. From \cite{TW95}, we know precisely when these algebras are representation finite.  The main goal of this section is to investigate the infinite case.

We give a classification of representation type for algebraically closed fields of characteristic $p$ and $G=C_{p^m}$.

\begin{theorem}\label[theorem]{thm:alg_closed_cohomological_rep_type}
  Let $k = \bar{k}$ be an algebraically closed field of characteristic $p$.  The cohomological Mackey algebra $\cohmu_k(C_{p^m})$ has
  \begin{enumerate}
    \item finite representation type if $m\leq 1$,
    \item tame (and infinite) representation type if $p^m=4$, and
    \item wild representation type otherwise.
  \end{enumerate}
\end{theorem}

We also show the following for arbitrary fields.

\begin{theorem}\label[theorem]{thm:arbitrary_field_cyclic_p_group_cohomological_rep_type}
  Let $k$ be a field of characteristic $p$.  Then the following hold.
  \begin{enumerate}
    \item The cohomological Mackey algebra $\cohmu_k(C_{p^m})$ has finite representation type if and only if $m \leq 1$.
    \item For $m=1$, there are precisely $5+ 4(p-2)$ isomorphism classes of indecomposable modules of $\cohmu_k(C_p)$. 
    \item If $m\geq 2$ and $p^m \neq 4$, then $\cohmu_k(C_{p^m})$ is wild.
  \end{enumerate}
\end{theorem}

It seems likely the indecomposables of $\cohmu_k(C_4)$ can be described in this case as well, but we do not pursue that here.

\subsection{Finite and infinite representation type}
Let $k$ be a (not necessarily algebraically closed) field of characteristic $p$. Thévenaz and Webb classified precisely when the Mackey algebra and the cohomological Mackey algebra are representation finite. Somewhat surprisingly, the conditions are the same for $\mu_k(G)$ and $\cohmu_k(G)$.

\begin{theorem}(\cite[Theorem~18.1]{TW95})\label[theorem]{thm:finite-rep-type-of-mackey-algebras}
  Let $k$ be a field of characteristic $p$ and let $C$ be a Sylow $p$-subgroup of $G$.  The following conditions are equivalent.
  \begin{enumerate}
    \item $\Mack_k(G)$ has finite representation type.
    \item The Mackey algebra $\mu_k(G)$ has finite representation type.
    \item $\CohMack_k(G)$ has finite representation type.
    \item The cohomological Mackey algebra $\cohmu_k(G)$ has finite representation type.
    \item The Hecke algebra $\hecke_k(G) = \End_{kG}\left( \bigoplus_{H \leq G} k[G/H] \right)$ 
    has finite representation type.
    \item $|C|=1$ or $|C|=p$ (or in other words $p^2 \! \not| \ |G|$).
  \end{enumerate}
\end{theorem}

Note that $(1) \iff (2)$ and $(3) \iff (4)$ follow immediately from the quiver perspective for Mackey functors, while $(4) \iff (5)$ follows from Yoshida's identification and duality.  The cohomological Mackey algebra is a quotient of the Mackey algebra so $(2) \implies (4)$. The bulk of the proof in \cite{TW95} is showing that (3) is equivalent to (6).  It makes use of the following fact.

\begin{proposition}[\cite{TWSimple90}]
  The functor $FP \colon kG\MMod \to \Mack_k(G)$, $V \mapsto FP_V$ sending a module to the fixed point Mackey functor is fully faithful.  Moreover, it is right adjoint to the forgetful functor $\ev_1 \colon \Mack_k(G) \to kG\MMod$, $M \mapsto M(G/e)$.
\end{proposition}

As described in \cite{TW95}, whenever $p^2$ divides the order of the group $G$, there is a subgroup $H$ isomorphic to either the cyclic group $C_{p^2}$ or the elementary abelian group $C_{p} \times C_{p}$.  Since $k[C_{p} \times C_{p}]$ has infinite representation type and $FP$ is a representation embedding, $\CohMack_k(G)$ has infinite representation type in this case.  Showing $\CohMack_k(G)$ is representation infinite in the case $H = C_{p^2}$ is the main task in showing $(3) \implies (6)$.

\bigskip

For ease of reference and to set up notation for later, we reproduce the bound quiver for $\cohmu_k(C_{p^m})$ and give explicit descriptions of the indecomposable $\cohmu_k(C_p)$-modules. 

Let $G = C_{p^m}$ be the cyclic group of order $p^m$ with subgroups given by $C_{p^i}$ for $0 \leq i \leq m$. In characteristic $p$, the group algebra $kG$ is isomorphic to $k[x]/(x^{p^m})$ and the transitive permutation modules are given by $k[x]/(x^{p^i})$ for $0 \leq i \leq m$. Homomorphisms between these modules are generated by the following maps
\begin{align*}
  a_i &\colon k[x]/(x^{p^i}) \ \ \ \to k[x]/(x^{p^{i-1}}) \quad \ \text{canonical projection}, \\
  b_i &\colon k[x]/(x^{p^{i-1}}) \to k[x]/(x^{p^i}) \qquad \text{multiplication by } x^{p^i - p^{i-1}},\\
  c_i &\colon k[x]/(x^{p^i}) \ \ \ \to k[x]/(x^{p^i}) \qquad \text{multiplication by } x,
\end{align*}
which satisfy following relations
\[
c_i^{p^i}=0 = a_i \circ b_i,\quad b_i \circ a_i = c_i^{(p^i - p^{i-1})},\quad a_i \circ c_i = c_{i-1} \circ a_i,\quad b_i \circ c_{i-1} = c_i \circ b_i 
\]
for $1\leq i \leq m$ with $c_{0}=0$.

Putting this together, and using that $\cohmu_k(G) \cong \hecke_k(G)^{\op}$, we get the following bound quiver.

\newcommand{\ctext}[1]{\text{\makebox[0pt]{$#1$}}}

\begin{proposition}[Cohomological Mackey algebra of $C_{p^m}$]\label[proposition]{prop:Cp^mCohMackQuiver}
Let $k$ be a field of characteristic $p$.  The cohomological Mackey algebra $\cohmu_k(C_{p^m})$ for $m\geq2$ is the path algebra of the quiver $Q$
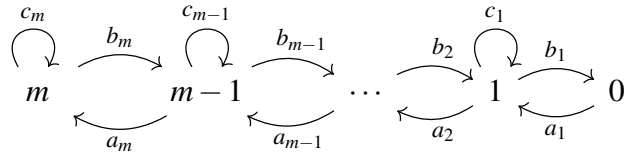
\begin{center}
  \begin{tikzcd}
    \phantom{-} m\phantom{1} \arrow[loop above, "c_m", out=120, in=60, looseness=5] \ar[r, bend left, "\ctext{b_m}"] & m-1 \ar[l, bend left, "\ctext{a_m}"] \ar[r, bend left, "\ctext{b_{m-1}}"] \arrow[loop above, "c_{m-1}", out=120, in=60, looseness=5] & \ar[l, bend left, "\ctext{a_{m-1}}"] \cdots \ar[r, bend left, "\ctext{b_2}"] & 1\ar[l, bend left, "\ctext{a_2}"] \ar[r, bend left, "\ctext{b_1}"] \arrow[loop above, "c_1", out=120, in=60, looseness=5] & 0 \ar[l, bend left, "\ctext{a_1}"]
  \end{tikzcd}
\end{center}
with relations\footnote{This is not a minimal set of relations, but convenient for later arguments.}
\[
\rho = \{ c_i^{p^i},\ b_i a_i,\ a_ib_i - c_i^{(p^i - p^{i-1})},\ c_i a_i - a_i c_{i-1},\ c_{i-1} b_i - b_i c_i \mid 1\leq i \leq m, c_0 = 0 \}.
\]
In the case $p=2$, the arrow $c_1$ is redundant because $c_1 = a_1b_1$, as in \cref{ex:C2constantquiver}.
\end{proposition}

For $m=1$, this reduces to the following.

\begin{proposition}[Cohomological Mackey algebra of $C_{p}$]\label[proposition]{prop:CpCohMackQuiver}
Let $k$ be a field of characteristic $p$.  The cohomological Mackey algebra $\cohmu_k(C_p)$ is the path algebra of the quiver $Q$
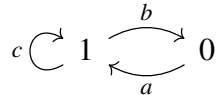
  \begin{center}
    \begin{tikzcd}
      1 \arrow[loop left, "c", out=210, in=150, looseness=5] \ar[r, bend left, "b"] & 0 \ar[l, bend left, "a"]
    \end{tikzcd}
  \end{center}
  with relations
  \[
  \rho = \{c^p,\ ba,\ ab - c^{p-1},\ ca,\ bc\},
  \]
  as in \cref{ex:Cpconstantquiver} for $k = \F_p$.
  
  In the case $p=2$, the arrow $c=ab$ is redundant and the cohomological Mackey algebra 
  $\cohmu_k(C_2)$ is the path algebra of the quiver $Q$
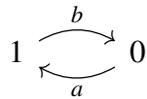
  \begin{center}
    \begin{tikzcd}
      1 \ar[r, bend left, "b"] & 0 \ar[l, bend left, "a"]
    \end{tikzcd}
  \end{center}
  with relation
  \[
  \rho = \{ba\},
  \]
  recovering \cref{ex:C2constantquiver} for $k = \F_2$.
\end{proposition}

From this description, we see that $\cohmu_k(C_p)$ is special biserial.  We can easily classify its modules using strings and bands by applying \cref{prop:counting_special_biserial_strings} with $n=p$.

\begin{proposition}
  In characteristic $p$, the algebra $\cohmu_k(C_p)$ is representation finite with $5 + 4(p-2)$ indecomposable modules.
\end{proposition}

At the prime $2$, we can compare with the classification of $\und{\F}_2$-modules from \cite{DHM}.
\begin{proposition}
At the prime $2$, the gentle algebra $\cohmu_k(C_2)$ is representation finite with $5$ indecomposable modules, coming from the strings
\[
e_0,\ e_1,\ a,\ b,\ ab.
\]
These correspond to the indecomposable Mackey functors 
\[
S_\bullet,\ S_\theta,\ H,\ H^{\op},\ F
\]
from \cite{DHM}, respectively.
\end{proposition}

For completeness, we give some simple constructions demonstrating that the cohomological Mackey algebra $\cohmu_k(C_{p^m})$ has infinite representation type whenever $m\geq 2$ and $k$ has characteristic $p$.

\begin{proposition}\label[proposition]{prop:cohomological_C4_rep_infinite}
  If $k$ has characteristic $2$, then $\cohmu_k(C_4)$ is representation infinite.
  \begin{proof}
    The algebra $\cohmu_k(C_4)$ is given by the quiver with relations 
    \begin{center}
      $\left.
      \begin{tikzcd}
        2 \arrow[loop left, "c", out=210, in=150, looseness=5] \ar[r, bend left, "b"] & 1 \ar[l, bend left, "a"] \ar[r, bend left, "e"] & 0 \ar[l, bend left, "d"]
      \end{tikzcd}
      \middle/ 
      \langle c^{4},\ ba,\ ed,\ ab-c^{2},\ ade-ca,\ deb-bc \rangle.
      \right.$
    \end{center}
    This is not special biserial because the arrow $b$ can be followed by either $a$ or $e$, since neither $ab$ nor $eb$ is a relation. 
    
    We construct infinitely many indecomposables. 
    Let $V =k[x]/(x^n)$ considered simply as an $n$-dimensional vector space with a distinguished endomorphism, multiplication by $x$. The representations
    \begin{center}
      \begin{tikzcd}
        V^2 \arrow[loop left, "C", out=210, in=150, looseness=5] \ar[r, bend left, "B"] & V^2 \ar[l, bend left, "A"] \ar[r, bend left, "E"] & V \ar[l, bend left, "D"],
      \end{tikzcd}
    \end{center}
    where 
    \[
    A=C= \begin{pmatrix}
      0&1\\0&0
    \end{pmatrix},\quad B= \begin{pmatrix}
      0&x\\0&0
    \end{pmatrix},\quad E = \begin{pmatrix}
      0 & 1
    \end{pmatrix},\quad \text{ and }\quad  D=\begin{pmatrix}
      1\\0
    \end{pmatrix}
    \]
     give an infinite family of non-isomorphic indecomposable representations.
  \end{proof}
\end{proposition}

\begin{proposition}
  If $k$ is a field of characteristic $p$ and $m\geq 2$, then $\cohmu_k(C_{p^m})$ has infinite representation type.
  \begin{proof}
    We already covered the case of $p^m = 4$ in \cref{prop:cohomological_C4_rep_infinite}, so we may assume $p^m > 4$. Again let $V = k[x]/(x^n)$. The representations supported on vertex $m$ and $m-1$ given by
    \begin{center}
      \begin{tikzcd}
        V^2 \arrow[loop left, "C", out=210, in=150, looseness=5] \ar[r, bend left, "B"] & V^2 \arrow[loop right, "D", out=30, in=-30, looseness=5] \ar[l, bend left, "A"]
      \end{tikzcd}
    \end{center}
    and $0$ elsewhere, where 
    \[
    A=C=D = \begin{pmatrix}
      0 & 1\\ 0 & 0
    \end{pmatrix}\quad \text{ and }\quad B = \begin{pmatrix}
      0 & x\\ 0 & 0
    \end{pmatrix},
    \]
     give an infinite family of non-isomorphic indecomposable representations.  Hence $\cohmu_k(C_{p^m})$ is representation infinite.
  \end{proof}
\end{proposition}

\subsection{Tame and wild representation type}
We now examine more closely the representation type of the cohomological Mackey algebra $\cohmu_k(C_{p^m})$ for $m \geq 2$ (still working over a field of characteristic $p$).  From \cref{thm:finite-rep-type-of-mackey-algebras}, we know $\cohmu_k(C_{p^m})$ is representation finite for $m < 2$ and infinite for $m \geq 2$.  We will show that when $m \geq 2$ and $p^m \neq 4$, the cohomological Mackey algebra is wild, and hence also the Mackey algebra (see \cref{sec:mackey_algebra}).  In the remaining case, $p^m = 4$, the cohomological Mackey algebra is tame over the algebraic closure $k = \bar{k}$.
\bigskip

We start by examining the case where $p^m = 4$.  From \cref{thm:finite-rep-type-of-mackey-algebras}, we know $\cohmu_k(C_4)$ has infinite representation type. In the case where $k=\bar{k}$ is algebraically closed, the situation is particularly nice. We can understand whether an algebra is wild by studying algo-geometric properties of its variety of representations. Recall from \cref{thm:dichotomy}, an algebra over $k = \bar{k}$ is either tame or wild, but not both.

\begin{proposition}\label[proposition]{prop:C4_tame}
  Let $k= \bar{k}$ be an algebraically closed field of characteristic $2$, then $\cohmu_k(C_4)$ is tame.
  
  \begin{proof}
    Consider the parametrized family of algebras $\tilde A(\xi, \zeta)$ given by the quiver with relations 
    \begin{center}
      $\left.
      \begin{tikzcd}
        2 \arrow[loop left, "c", out=210, in=150, looseness=5] \ar[r, bend left, "b"] & 1 \ar[l, bend left, "a"] \ar[r, bend left, "e"] & 0 \ar[l, bend left, "d"]
      \end{tikzcd}
      \middle/ 
      \langle c^{4},\ ba,\ ed,\  ab-\xi c^{2},\ \zeta ade-ca,\ \zeta deb-bc \rangle.
      \right.$
    \end{center}
    Note that for $\xi$ and $\zeta$ both nonzero $\tilde A(\xi, \zeta) \cong \cohmu_k(C_4)$. In particular, $\tilde A(\xi, \zeta)$ is isomorphic to $\cohmu_k(C_4)$ on a dense open set. We say that $\cohmu_k(C_4)$ \emph{degenerates} to the algebras outside this open set, e.g.\ $\cohmu_k(C_4)$ degenerates to $\tilde A(0,0)$. By \cite[Theorem~B]{CB95} if an algebra degenerates to a tame algebra then it must itself be tame. Observe that $\tilde A(0,0)$ is special biserial and hence tame \cite[Corollary~2.4]{WW85}. Thus we conclude that $\cohmu_k(C_4)$ is tame as well. 
  \end{proof}
\end{proposition}

We now turn to showing that for $m \geq 2$ and $p^m \neq 4$ the algebra $\cohmu_k(C_{p^m})$ is wild, for which we use Galois covers. 
Galois covers were introduced by Martínez-Villa and de la Peña in \cite{MP83}, and they give an important reduction technique in the representation theory of quivers. 
\begin{definition}[Galois covering \cite{MP83}]
  Let $(Q', \rho')$ be a (possibly infinite) quiver with relations. Let $G$ be a group of automorphisms of $(Q', \rho')$ which acts freely on the vertices of $Q'$. We can form a new quiver $(Q, \rho)$ whose vertices are the orbits of the vertices of $Q'$ and whose arrows are orbits of arrows of $Q'$. The relations in $\rho$ are simply the image of the relations in $\rho'$. In this case $(Q', \rho')$ is called a \emph{Galois covering} of $(Q, \rho)$ with Galois group $G$.
\end{definition}

The action of the group $G$ on $(Q', \rho')$ induces an action on $ kQ'/(\rho')\mmod$ (and on $ kQ'/(\rho')\MMod$) by ${}^{g}X(i) = X(g^{-1}i)$ for $i$ in $Q'_0$ and ${}^gX(\alpha) = X(g^{-1}\alpha)$ for $\alpha$ in $Q'_1$. 

There is a natural functor $F_\lambda \colon kQ'/(\rho')\mmod \to kQ/(\rho)\mmod$ with 
\[
F_\lambda(X)(\overline{i}) = \bigoplus_{g\in G} X(gi)
\]
and
\[
F_\lambda(X)(\overline{\alpha}) = \bigoplus_{g\in G} X(g\alpha).
\] 
Notice that $F_\lambda({}^gX) = F_\lambda(X)$.

\begin{lemma}\cite[Lemma~3.2]{Han01}\label[lemma]{lem:Han01}
  Let $(Q', \rho')$ be a Galois covering of $(Q, \rho)$ with Galois group $G$, and let $X$ and $Y$ be in $ kQ'/(\rho')\mmod$. Then there is an identification
  \begin{align*}
    \Hom(F_\lambda(X), F_\lambda(Y)) = \bigoplus_{g\in G} F_\lambda(\Hom(X, {}^g Y))
  \end{align*}
  that respects composition (identifying $F_\lambda(\Hom(X, {}^g Y))$ with $F_\lambda(\Hom({}^hX, {}^{hg} Y))$ as necessary).
\end{lemma}

The following corollary is found in \cite{Han01} but not stated explicitly, so we state and prove it below.

\begin{corollary}\label[corollary]{cor:toHanLemma}
  If $G$ is torsion free and $X \in  kQ'/(\rho')\mmod$ is indecomposable, then $F_\lambda(X)$ is also indecomposable.
  \begin{proof}
    Let $f \coloneqq \bigoplus f_g$ be an idempotent in $\End(F_\lambda(X))$. Then 
    \begin{align*}
      f = f^2 = \sum_{g,h} ({}^h f_g) \circ f_h
    \end{align*}
    Since $X$ is finite dimensional, it is only supported on finitely many vertices. Since $G$ is torsion free, $X$ and ${}^gX$ have different support, and hence are not isomorphic. Similarly, for all but finitely many $g$, the supports of $X$ and ${}^g X$ have no overlap so $\Hom(X, {}^g X) = 0$.

    Let $S$ be the (finite) set of $g$ for which $\Hom(X, {}^g X) \neq 0$. Let $\tilde X = \bigoplus_{g \in S} {}^g X$. Let $n$ be an integer such that $\rad^n \End(\tilde X) = 0$. We can think of ${}^h f_g$ as an endomorphism of $\tilde X$ through the composition $\tilde X \twoheadrightarrow {}^h X \to {}^{hg}X \hookrightarrow \tilde X$, where the last map is $0$ if $hg$ is not in $S$.

    When $g$ is not the identity, since $X$ and ${}^g X$ are not isomorphic, $f_g$ is a radical map. Letting $e$ be the identity of $G$, we observe 
    \begin{align*}
      f_e = f_e^2 + \sum_{g\neq e} ({}^g f_{g^{-1}}) \circ f_g.
    \end{align*}
    So $f_e$ is idempotent modulo the radical. This means that $f_e$ is either $1$ or $0$ modulo the radical, and replacing $f$ by $1-f$, we may assume $f_e$ is in the radical.
    Then 
    \begin{align*}
      f = f^n = \sum_{g_1, g_2, \cdots g_n} ({}^{g_1 g_2 \cdots g_{n-1}} f_{g_n}) \circ \cdots \circ f_{g_1}.
    \end{align*}
    Each term in the sum is a composition of $n$ radical maps, hence $0$. So we have shown that either $f$ or $1-f$ is $0$, and hence $F_\lambda(X)$ is indecomposable.
  \end{proof}
\end{corollary}

Han uses the above lemma to identify certain controlled wild algebras over algebraically closed fields, but the same proof can be used to detect wild algebras. We give an adapted proof here for the convenience of the reader and to demonstrate its generalization to arbitrary fields.

\begin{theorem}\cite[Theorem~3.3]{Han01}\label[theorem]{thm:wild_galois_cover}
  If $(Q', \rho')$ is a Galois cover of $(Q, \rho)$ with torsion free Galois group and $ kQ'/(\rho')\mmod$ is wild, then $kQ/(\rho)\mmod$ is wild as well.
  \begin{proof}
    Let $\Sigma = k\langle x, y \rangle$ be the free algebra on two generators. Since $kQ'/(\rho')\mmod$ is wild we can find a $kQ'/(\rho')$--$\Sigma$-bimodule $P$ supported on finitely many vertices that induces a representation embedding $P\otimes_\Sigma - \colon \Sigma\mmod \to kQ'/(\rho')\mmod$. Composing the representation embedding with $F_\lambda$, we claim we get a representation embedding to $kQ/(\rho)\mmod$. The composition is exact and by \cref{cor:toHanLemma} it preserves indecomposables. We just need to show it is essentially injective. 

    Let us assume $F_\lambda(P\otimes X) \cong F_\lambda(P\otimes Y)$ (omitting $\Sigma$ from the notation to avoid clutter). Without loss of generality, we may assume both $X$ and $Y$ are indecomposable. Let $f \colon F_\lambda(P\otimes X) \to F_\lambda(P\otimes Y)$ be an isomorphism and $f'$ its inverse. Then 
    \begin{align*}
      1_{P\otimes X} = f'_e \circ f_e + \sum_{g\neq e} ({}^g f'_{g^{-1}}) \circ f_g.
    \end{align*}
    Notice that $P\otimes X$ always has the same support (unless $X=0$), so the image of $f_g$ has smaller support for $g\neq e$.  Hence $f_g$ is not an isomorphism. That means that the compositions $({}^g f'_{g^-1}) \circ f_g$ are radical.  So $f_e$ is an isomorphism modulo the radical and hence an isomorphism. Finally, we may conclude $P\otimes X \cong P\otimes Y$, which implies $X\cong Y$ since $P \otimes - $ was a representation embedding.
  \end{proof}
\end{theorem}

The following uses results of \cite{Han02}, which are stated for algebraically closed fields.  Again, the arguments generalize to arbitrary fields.

\begin{proposition}\label[proposition]{prop:cohmu_cyclic_more_than_9_wild}
  Let $k$ be field of characteristic $p$.  The cohomological Mackey algebra $\cohmu_k(C_{p^m})$ is wild when $m \geq 2$ and $p^m >9$.
  \begin{proof}
    Recall from \cref{prop:Cp^mCohMackQuiver} that if $k$ is a field of characteristic $p$, then the cohomological Mackey algebra $\cohmu_k(C_{p^m})$ can be described by the following quiver
  \begin{center}
    \begin{tikzcd}
      \phantom{-} m\phantom{1} \arrow[loop above, "c_m", out=120, in=60, looseness=5] \ar[r, bend left, "\ctext{b_m}"] & m-1 \ar[l, bend left, "\ctext{a_m}"] \ar[r, bend left, "\ctext{b_{m-1}}"] \arrow[loop above, "c_{m-1}", out=120, in=60, looseness=5] & \ar[l, bend left, "\ctext{a_{m-1}}"] \cdots \ar[r, bend left, "\ctext{b_2}"] & 1\ar[l, bend left, "\ctext{a_2}"] \ar[r, bend left, "\ctext{b_1}"] \arrow[loop above, "c_1", out=120, in=60, looseness=5] & 0 \ar[l, bend left, "\ctext{a_1}"]
    \end{tikzcd}
  \end{center}
with relations
\[
\rho = \{ c_i^{p^i},\ b_i a_i,\ a_ib_i - c_i^{(p^i - p^{i-1})},\ c_i a_i - a_i c_{i-1},\ c_{i-1} b_i - b_i c_i \mid 1\leq i \leq m, c_0 = 0 \}.
\]
The quotient of $\cohmu_k(C_{p^m})$ by the idempotents $e_i$ corresponding to the vertices $i \in \{0, 1,\dots, m-2\}$ gives the quiver 
  
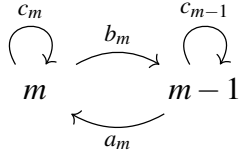
\begin{center}
      \begin{tikzcd}
        \phantom{-} m\phantom{1} \arrow[loop above, "c_m", out=120, in=60, looseness=5] \ar[r, bend left, "\ctext{b_m}"] & m-1 \ar[l, bend left, "\ctext{a_m}"] \arrow[loop above, "c_{m-1}", out=120, in=60, looseness=5] 
      \end{tikzcd}
  \end{center}
  with relations 
  \[
  \{c_m^{p^m},\ c_{m-1}^{(p^{m-1}-p^{m-2})},\ b_m a_m,\ a_mb_m - c_m^{(p^m - p^{m-1})},\ c_m a_m - a_m c_{m-1},\ c_{m-1} b_m - b_m c_m\}.
  \]
  Moreover, when $p^{m-1}-p^{m-2} \geq 3$, this surjects onto 
  \begin{center}
    $\left.
      \begin{tikzcd}
        \phantom{-} m\phantom{1} \arrow[loop above, "c_m", out=120, in=60, looseness=5] \ar[r, bend left, "\ctext{b_m}"] & m-1 \ar[l, bend left, "\ctext{a_m}"] \arrow[loop above, "c_{m-1}", out=120, in=60, looseness=5] 
      \end{tikzcd}
    \middle/ 
    \langle c_m^{3},\ c_{m-1}^{3},\ b_m a_m,\ a_mb_m,\ c_m a_m,\ a_m c_{m-1},\ c_{m-1} b_m - b_m c_m\rangle,
    \right.$
  \end{center}
  which is wild by \cite[Table~W~(33)]{Han02}. 
\end{proof}
\end{proposition}

We now employ similar techniques to tackle the remaining two cases with $m \geq 2$, namely $C_9$ and $C_8$.

\begin{proposition}\label[proposition]{prop:cohmu_C8_C9_wild}
  If $k$ is a field of characteristic $3$, then $\cohmu_k(C_9)$ is wild. If $k$ has characteristic $2$, then $\cohmu_k(C_8)$ is wild.
  \begin{proof}
    Both algebras have a quotient of the form 
    \begin{center}
      $\left.
        \begin{tikzcd}
          x \arrow[loop above, "c_x", out=120, in=60, looseness=7] \ar[r, bend left, "\ctext{b_x}"] & y \ar[l, bend left, "\ctext{a_x}"] \arrow[loop above, "c_{y}", out=120, in=60, looseness=7] \ar[r, "a_y"] & z 
        \end{tikzcd}
      \middle/ 
      \langle c_x^{4},\ c_{y}^{2},\ b_x a_x,\ a_xb_x,\ c_x a_x,\ a_x c_{y},\ c_{y} b_x - b_x c_x\rangle,
      \right.$
    \end{center}
    so it is enough to show that this is wild. This has a Galois covering with Galois group $\mathbb Z^2$ which looks like 
\[\begin{tikzcd}[/tikz/cells={/tikz/nodes={shape=asymmetrical
  rectangle,text width=1ex,text height=1ex,text depth=0.3ex,align=center}}]
	\mathrlap{\ddots} & \phantom{2} & \vdots & \phantom{2} & \mathrlap{\vdots} & \phantom{2} & \mathrlap{\iddots} \\
	\phantom{2} && 2 && 1 & 0 & \phantom{2} \\
	\mathrlap{\cdots} & 2 && 1 & 0 & 2 & \mathrlap{\cdots} \\
	\phantom{2} && 1 & 0 & 2 && \phantom{2} \\
	\mathrlap{\iddots} & \phantom{2} & \mathrlap{\vdots} & \phantom{2} & \mathrlap{\vdots} & \phantom{2} & \mathrlap{\ddots}
	\arrow[from=1-2, to=2-3]
	\arrow[from=1-4, to=2-3]
	\arrow[from=1-4, to=2-5]
	\arrow[from=1-6, to=2-5]
	\arrow[from=2-1, to=3-2]
	\arrow[from=2-3, to=3-2]
	\arrow[from=2-3, to=3-4]
	\arrow[from=2-5, to=2-6]
	\arrow[from=2-5, to=3-4]
	\arrow[from=2-5, to=3-6]
	\arrow[from=2-7, to=3-6]
	\arrow[from=3-2, to=4-1]
	\arrow[from=3-2, to=4-3]
	\arrow[from=3-4, to=3-5]
	\arrow[from=3-4, to=4-3]
	\arrow[from=3-4, to=4-5]
	\arrow[from=3-6, to=4-5]
	\arrow[from=3-6, to=4-7]
	\arrow[from=4-3, to=4-4]
	\arrow[from=4-3, to=5-2]
	\arrow[from=4-3, to=5-4]
	\arrow[from=4-5, to=5-4]
	\arrow[from=4-5, to=5-6]
\end{tikzcd}\]
  with appropriate relations. It contains as a convex subquiver
  \[\begin{tikzcd}
    &&&& 2 \\
    & 1 && 2 && 1 \\
    1 && 2 && 1 && 0 \\
    & 2
    \arrow[from=1-5, to=2-4]
    \arrow[from=1-5, to=2-6]
    \arrow[from=2-2, to=3-1]
    \arrow[from=2-2, to=3-3]
    \arrow[from=2-4, to=3-3]
    \arrow[from=2-4, to=3-5]
    \arrow[from=2-6, to=3-5]
    \arrow[from=2-6, to=3-7]
    \arrow[from=3-1, to=4-2]
    \arrow[from=3-3, to=4-2]
  \end{tikzcd}\]
  with commutativity relations. This is a concealed algebra of type $\tilde{\tilde D}_7$, which is known to be wild \cite{ASS07}.  More precisely, we can get a representation embedding of $\Sigma = k\langle x, y \rangle$ using the following bimodule.
  \[\begin{tikzcd}[ampersand replacement = \&, row sep=3.5em, column sep=3.5em]
    \&\&\&\& \Sigma^2 \\
    \& \Sigma^2 \&\& \Sigma^4 \&\& \Sigma^2 \\
    \Sigma^2 \&\& \Sigma^4 \&\& \Sigma^2 \&\& \Sigma. \\
    \& \Sigma^2
    \arrow[from=1-5, to=2-4, swap]{}{
      \begin{pmatrix}
       1 & 0 \\0 & y\\ 0 & 1\\ 1 & x
      \end{pmatrix} 
     } 
    \arrow[from=1-5, to=2-6, equal]
    \arrow[from=2-2, to=3-1, equal]
    \arrow[from=2-2, to=3-3]{}{\begin{pmatrix}
      I_2 \\ I_2
    \end{pmatrix}}
    \arrow[from=2-4, to=3-3, equal]
    \arrow[from=2-4, to=3-5, swap]{}{\begin{pmatrix}
       0 & I_2
     \end{pmatrix} }
    \arrow[from=2-6, to=3-5]{}{\begin{pmatrix}
      0 & 1 \\ 1 & x 
     \end{pmatrix} }
    \arrow[from=2-6, to=3-7]{}{\begin{pmatrix}0&1\end{pmatrix}}
    \arrow[from=3-1, to=4-2, swap]{}{\begin{pmatrix}
      I_2 & 0
     \end{pmatrix} }
    \arrow[from=3-3, to=4-2]{}{\begin{pmatrix}
      I_2 & 0 
     \end{pmatrix} }
  \end{tikzcd}\]
  \end{proof}
\end{proposition}

Collecting the results from this section, we have proved \cref{thm:alg_closed_cohomological_rep_type} and \cref{thm:arbitrary_field_cyclic_p_group_cohomological_rep_type}.

\section{Derived representation type}\label[section]{sec:derived_wild_groups}

Motivated by equivariant homotopy theory, we are interested in the derived category of various algebras. 
We will show that, for most groups, it is `hopeless' to classify the perfect complexes over the cohomological Mackey algebra $\cohmu_k(G)$ for $k$ a field of characteristic $p$. More formally, we will investigate when $\cohmu_k(G)$ is derived wild (defined below).  As explained in \cref{sec:Eilenberg-Maclane-spectra}, when $\cohmu_k(G)$ is derived wild, this will mean it is also hopeless to classify the compact objects of the homotopy category of $G$-equivariant $H\und{k}$-modules.

In this section, we recall the relevant definitions for derived representation type.  We then show the derived wildness of some group algebras, following \cite{BDF09}.  In \cref{sec:derived-wildnes-cohmackey}, we use these results to show the derived wildness of $\cohmu_k(G)$ in a number of cases.

\subsection{Derived representation type of algebras}\label[section]{sec:derived_reptype}
In \cref{sec:reptype}, we discussed the representation type of an algebra; we now extend those ideas to the perfect derived category of an algebra.  We saw that an algebra $\Lambda$ is considered wild if it admits a representation embedding from $\Gamma\mmod$ for every finite dimensional algebra $\Gamma$. We wish to declare the same for the perfect derived category: an algebra $\Lambda$ is derived wild if there is a ``representation embedding'' $\Gamma\mmod \to \mathcal{D}^{\perf}(\Lambda)$. We need only define what a representation embedding should mean in this context.

Like representation type, we can easily define derived finite and derived wild for arbitrary fields. However, derived tame is only defined for algebraically closed fields. 

Let $k$ be an arbitrary field.  Let $\Lambda$ be a $k$-algebra and let $\Lambda\mproj \subseteq \Lambda\MMod$ denote the full subcategory of finitely-generated projective left $\Lambda$-modules.  Let $K^b(\Lambda\mproj)$ denote the (naive) homotopy category of bounded complexes (both above and below) of finitely  generated projective $\Lambda$-modules, with morphisms given by chain maps modulo chain homotopy. Let $K^{b, -}(\Lambda\mproj)$ denote the category of bounded above complexes of finitely generated projectives with bounded cohomology. Then we have equivalences
\[
\mathcal{D}^b(\Lambda) \cong K^{b, -}(\Lambda\mproj) \quad \text{and} \quad \mathcal{D}^{\perf}(\Lambda) \simeq K^b(\Lambda\mproj),
\]
so we can work in the homotopy category when convenient.

Recall that a complex of projectives is called minimal if the image of each differential is contained in the radical of its codomain. A minimal complex of projectives has the property that it appears as a direct summand in any complex of projectives it is homotopy equivalent to. This is what motivates the name \emph{minimal}.

When $\Lambda$ is a finite dimensional algebra, any complex in $K^b(\Lambda\mproj)$ is homotopy equivalent to a minimal complex. This is especially useful for us, because a morphism between minimal complexes is an isomorphism in $K^b(\Lambda\mproj)$ if and only if it is an isomorphism in $\Ch(\Lambda\mmod)$. Similarly, since the direct sum of minimal complexes is minimal, a minimal complex is indecomposable in $K^b(\Lambda\mproj)$ if and only if it is in $\Ch(\Lambda\mmod)$. This means that as long as we work with minimal complexes, we can for the most part think of morphisms in $\Ch(\Lambda\mmod)$ instead of thinking of them up to homotopy.

\bigskip

We now define derived representation type analogously to representation type.  The definitions of derived wild and derived tame are due to \cite{BD03} in the more general setting of locally finite dimensional categories.  Various versions also appear for algebras in other sources, including \cite{Dro04} and \cite{BM03}.  Our main focus is on perfect complexes because of their connection to compact objects in equivariant homotopy theory.  Over algebraically closed fields, similar definitions to those below, replacing $\mathcal{D}^{\perf}(\Lambda)$ with $\mathcal{D}^{b}(\Lambda)$, give an equivalent classification as shown in \cite{Bau07}. 

\begin{definition}[Derived finite]
  Let $k$ be a field and let $\Lambda$ be a $k$-algebra.  We say $\Lambda$ is \emph{derived finite} if there exist up to shift only finitely many isomorphism classes of indecomposables $X^\bullet_1, \dots, X^\bullet_n \in \mathcal{D}^{\perf}(\Lambda)$.  That is, every complex $C^\bullet \in \mathcal{D}^{\perf}(\Lambda)$ is isomorphic to a finite direct sum of shifts of the complexes $\{X^\bullet_1, \dots, X^\bullet_n\}$.
\end{definition}

\begin{example}
  Let $\Lambda = k$.  There is a single indecomposable complex with $k$ concentrated in degree $0$.
\end{example}

More generally, semisimple algebras are derived finite.

\begin{example}\label[example]{ex:semisimple_derived_finite}
  Let $\Lambda$ be a finite dimensional semisimple $k$-algebra.  Then the simples are projective. Using Smith normal form, the indecomposable complexes correspond to the simples concentrated in degree $0$.
\end{example}

By \cref{prop:invertible_characteristic} and \cref{cor:cohmu_invertible_characteristic}, as in \cref{ex:semisimple_derived_finite}, both the cohomological and the usual Mackey algebra are derived finite.

\begin{corollary}\label[corollary]{cor:invertible_char_derived_finite}
  If $|G|$ is invertible in the field $k$, both $\cohmu_k(G)$ and $\mu_k(G)$ are derived finite.
\end{corollary}

\begin{corollary}
  If $|G|$ is invertible in the field $k$, there are finitely many indecomposable compact $H\und{k}$-modules.  Moreover, there are finitely many indecomposable compact $H\und{A}_k$-modules. 
\end{corollary}

\begin{definition}[Derived tame]
  Let $k = \bar{k}$ be an algebraically closed field.  We say that $\Lambda$ is \emph{derived tame} if, for each cohomology dimension vector $(d_i)_{i \in \Z}$ there are a finite number of bounded complexes of $\Lambda$--$k[x]$-bimodules $C^\bullet_1, \dots, C^\bullet_n$ such that each $C^i_j$ is free and of finite rank as right $k[x]$-modules, and such that every indecomposable $X^\bullet \in \mathcal{D}^{\perf}(\Lambda)$ with $\dim H^i(X^\bullet) = d_i$ is isomorphic to $C^\bullet_j \otimes_{k[x]}S$ for some $j$ and some simple $k[x]$-module $S$. 
\end{definition}

Over arbitrary fields, there is a notion of derived discrete\footnote{There is a notion of discrete for representation type.  Over algebraically closed fields it turns out to be equivalent to finite representation type by the second Brauer--Thrall conjecture as shown in \cite{Bau85}.} that is, in a sense, somewhere between derived finite and derived tame.

\begin{definition}[Derived discrete]
  Let $k$ be an arbitrary field.  We say that $\Lambda$ is \emph{derived discrete} if for every cohomology dimension vector $(d_i)_{i \in \Z}$, there is up to isomorphism a finite number of indecomposables $X^\bullet \in \mathcal{D}^{\perf}(\Lambda)$ with $\dim H^i(X^\bullet) = d_i$.
\end{definition}

\begin{remark}
  Over a finite field $k$, every finite dimensional $k$-algebra is derived discrete.  The definition of derived discrete is more meaningful over an infinite field.
\end{remark}

\begin{example}
  Let $k$ be a field. Then $\Lambda = k[x]/(x^2)$ is derived discrete, but not derived finite.
\end{example}

Inspired by the characterization of wildness in \cref{prop:embedding_is_tensorproduct}, Bekkert and Drozd gave a definition in \cite{BD03} similar to the following (again stated more generally in the language of locally finite dimensional categories). We give the definition as it appears in \cite{Dro04}.
\begin{definition}[Derived wild]
  We say that $\Lambda$ is \emph{derived wild} if for any finite dimensional\footnote{Technically \cite{Dro04} requires any finitely generated algebra here.  However, the proof of \cref{thm:wildfreealgebra} showing $k\langle x,y \rangle$ is wild applies equally well to finitely generated algebras.} algebra $\Gamma$ there is a complex of finitely generated projective bimodules $P^\bullet$ in  $K^{b}(\Lambda \otimes \Gamma^{\rm{op}}\mproj)$ such that for any $M, N$ in $\Gamma\mmod$ we have 
  \begin{enumerate}
    \item $P^\bullet \otimes_\Gamma M \cong P^\bullet \otimes_\Gamma N$ implies $M \cong N$, and
    \item if $P^\bullet \otimes_\Gamma M$ is indecomposable then $M$ is as well.
  \end{enumerate}
  We call $P^\bullet$ a \emph{strict family of $\Lambda$-complexes over $\Gamma$}.
\end{definition}

\begin{remark}
  As in \cref{{prop:embedding_from_wild_implies_wild}}, 
  it is enough to establish a strict family over any wild algebra.
\end{remark}

 Observe that, using minimal complexes, the technique described in \cref{prop:key_technique} can also be used to demonstrate an algebra is derived wild.  
 
 Bekkert--Drozd showed that, over an algebraically closed field, there is a derived version of the tame-wild dichotomy.

 \begin{theorem}(Derived tame-wild dichotomy \cite{BD03,Dro04})\label[theorem]{thm:derived_dichotomy}
  Let $k = \bar{k}$ be an algebraically closed field.  A finite dimensional $k$-algebra is either derived tame or derived wild, and not both.
\end{theorem}

 Much like wild representation type, the definition of derived wild implies it is hopeless to classify all indecomposable perfect complexes. For a non-example, we refer to the classification of $\mathcal{D}^{\perf}(\cohmu_{\F_2}(C_2))$ from \cite{DHM}.  Although the cohomological Mackey algebra is not derived finite, we can describe all indecomposable perfect complexes in a meaningful way.

 \begin{example}\label[example]{ex:CohC2DerivedTameDHM}
  It was shown in 
  \cite[Theorem~4.5]{DHM} and \cite[Corollary~4.6]{DHM} that the perfect derived category of $\cohmu_{\F_2}(C_2)$ has three families of isomorphism classes of indecomposables called $A_m$, $B_r$, and $H(n)$ for $m,r \geq 0$ and $n \in \Z$.  This classification can be used to show $\cohmu_{\F_2}(C_2)$ is not derived wild. If we replace $\F_2$ with an infinite field $k$ of characteristic $2$, then $\cohmu_k(C_2)$ is derived discrete, and thus not derived wild. We return to this example in \cref{ex:C2_gentle_homotopy_strings}. 
 \end{example}

 Our main objective will be to show the cohomological Mackey algebra over an arbitrary field is derived wild in a number of cases. 
  Our arguments will make use of the fact that algebras inherit derived wildness from idempotent subalgebras.

\begin{definition}[Idempotent subalgebra]
  We say that $\Gamma$ is an \emph{idempotent subalgebra} of $\Lambda$ if $\Gamma = e\Lambda e$ for an idempotent $e \in \Lambda$.
\end{definition}

\begin{proposition}\label[proposition]{prop:idempotent_subalgebra_derived_wild}
   If an idempotent subalgebra $\Gamma = e\Lambda e$ of $\Lambda$ is derived wild, then so is $\Lambda$.
  \begin{proof}
    The functor $\Lambda e \otimes_\Gamma - \colon \Gamma\mmod \to \Lambda\mmod$ is fully faithful and maps projective modules to projective modules. Thus, if $P^\bullet$ is a strict family of $\Gamma$-complexes over some other finite dimensional algebra, then $\Lambda e \otimes_\Gamma P^\bullet$ is a strict family of $\Lambda$-complexes.  
  \end{proof}
\end{proposition}

\subsection{Some derived wild group algebras}
We now explicitly show derived wildness for certain algebras, including the group algebras $k[C_{p^m}]$ and $k[C_p \times C_p]$.  We will use these in the next section to investigate the derived representation type of the cohomological Mackey algebra $\cohmu_k(G)$.

The first algebra we consider generalizes the group algebra of a cyclic group. Recall, when $k$ has characteristic $p$, the group algebra $k[C_{p^m}]$ is isomorphic to $k[c]/(c^{p^m})$.  The following theorem showing the truncated polynomial algebra is derived wild (over algebraically closed fields) was proven in \cite[Theorem~A]{BDF09} using the theory of bocses. We unpack their proof here for several reasons: to avoid such complicated machinery, to generalize to arbitrary fields, and to correct a small error for the case $k[c]/(c^3)$. 

\begin{proposition} \label[proposition]{prop:wild_cyclic_group}
  Let $k$ be a field. For $n \geq 2$ the algebra $R = k[c]/(c^{n+1})$ is derived wild.
  \begin{proof}
    Recall from \cref{prop:A1tildetilde_wild}, that the below quiver is wild.
    \begin{center}
      \begin{tikzcd}
        1 \ar[r, bend left, "a"] \ar[r, bend right, swap, "b"] & 2 \ar[r, "w"] & 3
      \end{tikzcd}
    \end{center}
    Let $\Gamma$ be the path algebra of this quiver and write $P_i \coloneqq R \otimes e_i\Gamma$. Consider the following complex of $R$--$\Gamma$-bimodules.\footnote{This complex is adapted from a complex in the proof of \cite[Theorem~A]{BDF09}.} 
    \begin{center}
      \begin{tikzcd}[ampersand replacement = \&, column sep=3.5em, row sep = -0.5em]
        \& P_1 \ar[dr]{}{\makebox[0pt]{$\begin{pmatrix}-c^{n-1}\\c^{n-1}\end{pmatrix}$}} 
        \&\&\&
        P_1 \ar[dr]{}{\makebox[0pt]{$c^{n-1} \otimes a$}}\\
        P_1 \ar[ur]{}{\makebox[0pt]{$c^{n}$}} \ar[dr, swap]{}{\makebox[0pt]{$c^{n-1}$}} 
        \& \oplus \& 
        P_1^2 \ar[r]{}{\makebox[0pt]{$\begin{pmatrix}c^n&0\\ 0 & c^n \end{pmatrix}$}} 
        \&
        P_1^2 \ar[ur]{}{\makebox[0pt]{$\begin{pmatrix}c^n&0\end{pmatrix}$}} 
        \ar[dr, swap]{}{\makebox[0pt]{$\begin{pmatrix}0&c^{n-1}\end{pmatrix}$}} 
        \& \oplus \& 
        P_2 \ar[r]{}{\makebox[0pt]{$c^n \otimes w$}} \& P_3\\
        \& P_1 \ar[ur, swap]{}{\makebox[0pt]{$\begin{pmatrix}c^{n}\\0\end{pmatrix}$}} 
        \&\&\&
        P_1 \ar[ur, swap]{}{\makebox[0pt]{$c^n \otimes b$}} 
      \end{tikzcd}
    \end{center}
    We claim that this is a strict family of $R$-complexes over $\Gamma$. To see this let 
    \begin{center}
      \begin{tikzcd}
        k^r \ar[r, bend left, "A"] \ar[r, bend right, swap, "B"] & k^s \ar[r, "W"] & k^t
      \end{tikzcd}
    \end{center}
    be a representation of $\Gamma$. After tensoring with the above complex, the resulting complex of $R$-modules is $C^\bullet$ given by
    \begin{center}
      \begin{tikzcd}[ampersand replacement = \&, column sep = 3.5em]
        R^r \ar{r}{\begin{pmatrix}c^{n}\\c^{n-1}\end{pmatrix}} \& 
        R^{2r} \ar{r}{\begin{pmatrix}-c^{n-1} & c^{n}\\c^{n-1} & 0\end{pmatrix}} \&
        R^{2r} \ar[r]{}{\makebox[0pt]{$\begin{pmatrix}c^n&0\\ 0 & c^n \end{pmatrix}$}} \&
        R^{2r} \ar{r}{\begin{pmatrix}c^{n} & 0\\ 0 & c^{n-1}\end{pmatrix}} \&
        R^{2r} \ar{r}{\begin{pmatrix}c^{n-1}A & c^{n}B\end{pmatrix}} \&
        R^s \ar{r}{c^n W} \&
        R^t.
      \end{tikzcd}
    \end{center}
    A morphism $\Phi \colon C^\bullet \to \hat{C}^\bullet$ of two such complexes is of the following form.
    \begin{center}
      \begin{tikzcd}[ampersand replacement = \&, column sep = 3.5em, row sep = 4.5em]
        R^r \ar{r}{\begin{pmatrix}c^{n}\\c^{n-1}\end{pmatrix}} 
        \ar{d}{\Phi_1}\& 
        R^{2r} \ar{r}{\begin{pmatrix}-c^{n-1} & c^{n}\\c^{n-1} & 0\end{pmatrix}}  
        \ar{d}{\begin{pmatrix} \Phi_{21} & \Phi_{22}\\ \Phi_{23} & \Phi_{24} \end{pmatrix}}\&
        R^{2r} \ar[r]{}{\makebox[0pt]{$\begin{pmatrix}c^n&0\\ 0 & c^n \end{pmatrix}$}} 
        \ar{d}{\begin{pmatrix} \Phi_{31} & \Phi_{32}\\ \Phi_{33} & \Phi_{34} \end{pmatrix}}\&
        R^{2r} \ar{r}{\begin{pmatrix}c^{n} & 0\\ 0 & c^{n-1}\end{pmatrix}} 
        \ar{d}{\begin{pmatrix} \Phi_{41} & \Phi_{42}\\ \Phi_{43} & \Phi_{44} \end{pmatrix}}\&
        R^{2r} \ar{r}{\begin{pmatrix}c^{n-1}A & c^{n}B\end{pmatrix}} 
        \ar{d}{\begin{pmatrix} \Phi_{51} & \Phi_{52}\\ \Phi_{53} & \Phi_{54} \end{pmatrix}}\&
        R^s \ar{r}{c^n W} 
        \ar[d]{}{\Phi_6}\&
        R^t \ar{d}{\Phi_7}\\
        R^{\hat{r}} \ar[swap]{r}{\begin{pmatrix}c^{n}\\c^{n-1}\end{pmatrix}} \& 
        R^{2\hat{r}} \ar[swap]{r}{\begin{pmatrix}-c^{n-1} & c^{n}\\c^{n-1} & 0\end{pmatrix}}  \&
        R^{2\hat{r}} \ar[r, swap]{}{\makebox[0pt]{$\begin{pmatrix}c^n&0\\ 0 & c^n \end{pmatrix}$}} \&
        R^{2\hat{r}} \ar[swap]{r}{\begin{pmatrix}c^{n} & 0\\ 0 & c^{n-1}\end{pmatrix}} \&
        R^{2\hat{r}} \ar[swap]{r}{\begin{pmatrix}c^{n-1}\hat A & c^{n}\hat B\end{pmatrix}} \&
        R^{\hat s} \ar[swap]{r}{c^n \hat W} \&
        R^{\hat t}
      \end{tikzcd}
    \end{center}
    Using the strategy of \cref{prop:key_technique} we need to construct from $\Phi$ a homomorphism $\Delta_{C,\hat{C}}(\Phi)$ of the underlying $\Gamma$-modules. Our candidate for this is the matrices $(\Phi_{51}, \Phi_6, \Phi_7)$ reduced modulo $(c)$. Thus we need to show that these define a homomorphism of $\Gamma$-modules and satisfy the criteria of \cref{prop:key_technique}.

    Setting up the equations for the commutativity of the above diagram and simplifying gives us in particular that 
    \begin{align*}
      c^n \Phi_1 = c^n \Phi_{21} = c^n \Phi_{24} = c^n \Phi_{31} = c^n \Phi_{34} &= c^n \Phi_{41} = c^n \Phi_{44} = c^n \Phi_{51} = c^n \Phi_{54},\\
      c^n \Phi_{52} = c^n \Phi_{42} = c^n \Phi_{32} &= c^n \Phi_{22} = 0,\\
      \text{and so also}&\\
      c^{n-1} \Phi_{52}=0.&
    \end{align*}
    Looking just at commutativity in the two rightmost squares we see that 
    \begin{align*}
      c^n \hat A \Phi_{51} &= c^n \Phi_{6} A\\
      c^n \hat B \Phi_{51} &= c^n \Phi_{6} B\\
      c^n \hat W \Phi_6 &= c^n \Phi_7 W.
    \end{align*}
    These equations tell us that $\Delta_{C,\hat{C}}(\Phi)=(\Phi_{51}, \Phi_6, \Phi_7)$ gives a homomorphism of $\Gamma$-modules when reduced modulo $(c)$. Let $\overline{\Phi}$ be the reduction $\Phi$ modulo $(c)$. If $\Phi$ is idempotent, then $\overline{\Phi}$ is also idempotent. Since $(c)$ is a nilpotent ideal $\Phi$ is an isomorphism if and only if $\overline{\Phi}$ is. From the above equations we can see that $\overline{\Phi}$ consists of lower triangular matrices. Lower triangular matrices are invertible if and only if the diagonal entries are. From the first set of equations above, all the diagonal entries in $\Phi$ are equal modulo $(c)$ except for $\Phi_6$ and $\Phi_7$.  Thus we have that $\Phi$ is an isomorphism if and only if $\Phi_{51}$, $\Phi_6$ and $\Phi_7$ are isomorphisms modulo $(c)$. Since a lower triangular matrix being idempotent implies its diagonal entries are idempotent, $\Phi$ idempotent implies that $\Phi_{51}$, $\Phi_6$ and $\Phi_7$ are idempotent modulo $(c)$. Hence all the criteria of \cref{prop:key_technique} are satisfied and this is a strict family of complexes.
  \end{proof}
\end{proposition}

Next we investigate algebras similar to the group algebra of the product of cyclic groups. Specifically when $k$ has characteristic $p$, the group algebra $k[C_p \times C_p]$ is isomorphic to $k[s,t]/(s^p, t^p)$.  The following is a special case of \cite[Lemma~3.2]{BDF09}, which again is stated for algebraically closed fields.  We demonstrate the details of the proof over arbitrary fields.

\begin{proposition}\cite[Lemma~3.2]{BDF09} \label[proposition]{prop:wild_CpxCp}
  Let $k$ be a field. If $n\geq 2$ then the algebra $R \coloneqq k[s, t]/(s^n, t^n)$ is derived wild. 
  \begin{proof}
    Again we let $\Gamma$ be the path algebra of the quiver from \cref{prop:A1tildetilde_wild}, and let $P_i$ denote $R \otimes e_i\Gamma$. Then the complex
    \begin{center}
      \begin{tikzcd}[ampersand replacement = \&, column sep = 5em]
        P_1 \ar[r]{}{s\otimes a + t\otimes b} \& P_2 \ar[r]{}{(st)^{n-1} \otimes w} \& P_3
      \end{tikzcd}
    \end{center} 
    defines a strict family of complexes. To see this we take two representations of $\Gamma$, $M_1$ and $M_2$ and consider a morphism $\Phi \colon P^\bullet \otimes_\Gamma M_1 \to P^\bullet \otimes_\Gamma M_2$. This will be of the following form. 
    \begin{center}
      \begin{tikzcd}
        R^p \ar[r, "sA + tB"] \ar[d, "\Phi_1"] & R^q \ar[r, "(st)^{n-1}W"] \ar[d, "\Phi_2"] & R^r \ar[d, "\Phi_3"] \\
        R^{\hat p} \ar[r, "s\hat A + t\hat B"] & R^{\hat q} \ar[r, "(st)^{n-1}\hat W"] & R^{\hat r} 
      \end{tikzcd}
    \end{center}
    The commutativity relations yield 
    \begin{align*}
      s \hat A \Phi_1 + t \hat B \Phi_1 = s\Phi_2 A + t\Phi_2 B\quad\quad
      \text{and}\quad\quad
      (st)^{n-1} \hat W \Phi_2 = (st)^{n-1} \Phi_3 W,
    \end{align*}
    where the entries in $\Phi_i$ are elements of $R$. Consider the first equation modulo $(s,t)^2$ to kill off higher powers. Then because $s$ and $t$ are linearly independent, we have for example $s \hat A \Phi_1 = s\Phi_2 A$.  Hence we get that $\hat A \Phi_1 = \Phi_2 A$ and $\hat B\Phi_1 = \Phi_2 B$ modulo $(s,t)$. Similarly, $\hat W \Phi_2 = \Phi_3 W$ modulo $(s,t)$.  So reducing $\Phi$ modulo $(s,t)$ gives a homomorphism $\Delta (\Phi) \colon M_1 \to M_2$ as in \cref{prop:key_technique}.
  \end{proof}
\end{proposition}
   
\begin{remark}
  Over algebraically closed fields, the derived representation type can only be more complicated than the representation type.  
  As in \cref{ex:CohC2DerivedTameDHM}, an algebra may be representation finite and derived tame.  
  Even more extreme, as we see in \cref{prop:wild_cyclic_group}, an algebra may be representation finite and derived wild.
 \end{remark}

\subsection{Derived representations of gentle algebras}  Over an algebraically closed field, Bekkert--Merklen showed gentle algebras are derived tame \cite{BM03}.  In fact, over any field, the perfect complexes of a gentle algebra can be described by homotopy strings and homotopy bands.

A perfect complex is a sequence of homomorphisms between projectives composing to zero and a homotopy string is designed to encapsulate that data. Strings give a basis for homomorphisms between indecomposable projectives, so they form the letters in homotopy strings. We give some of the relevant definitions below. See \cite{BM03} or \cite[Section~2.1.2]{OPS25}. See also \cite{Lak16} for more detailed exposition.

Let $k$ be a field and $\Lambda \cong kQ/(\rho)$ be the path algebra of a gentle $k$-algebra  A direct (resp.\ inverse) string $\omega$ is called a \emph{direct homotopy letter (resp.\ inverse homotopy letter)}. A \emph{homotopy string} $\sigma$ is a (possibly trivial) reduced walk given by a sequence of homotopy letters $\sigma = \omega_r \dots \omega_1$, satisfying the following:
\begin{enumerate}
\item if $\omega_{i+1}$ and $\omega_i$ are both direct, then $\omega_{i+1}\omega_i \in (\rho)$, and
\item if $\omega_{i+1}$ and $\omega_i$ are both inverse, then $\omega_{i}^{-}\omega_{i+1}^{-} \in (\rho)$.
\end{enumerate} 
 A \emph{homotopy band} $\sigma$ is a homotopy string with an equal number of direct and inverse homotopy letters such that the walk forms a cycle, and such that $\sigma \neq \tau^m$ for some homotopy string $\tau$ and $m \geq 2$.

From a homotopy string, we can form a perfect complex called a string complex by taking an indecomposable projective whenever the walk passes a vertex and using homotopy letters to define maps between them.  Note that because of our choice of conventions for path notation and left modules, composition of homotopy letters corresponds to composition of maps in the reverse order (see \cite[Remark~4.1.3]{Lak16}).  There is also a notion of a band complex and together with string complexes these form all the indecomposable perfect complexes.

\begin{theorem}[\cite{BM03}]
  Let $k$ be an arbitrary field and let $\Lambda \cong kQ/(\rho)$ be the path algebra of gentle algebra.  The indecomposable objects in $\mathcal{D}^{\perf}(\Lambda)$ are determined by the homotopy strings and homotopy bands, up to inverse.  Each homotopy string corresponds to an indecomposable complex, while the homotopy bands correspond to families of complexes.
\end{theorem}

Our main example of a gentle algebra is $\cohmu_k(C_2)$.

\begin{example}\label[example]{ex:C2_gentle_homotopy_strings}
Consider the bound quiver
   \begin{center}
   $\left.
    \begin{tikzcd}
      1 \ar[r, bend left, "b"] & 0 \ar[l, bend left, "a"]
    \end{tikzcd}
    \middle/ 
    \langle ba \rangle
    \right.$
  \end{center}
  for the gentle algebra from \cref{ex:C2_constant_gentle} and \cref{ex:C2constant_strings} (describing $\cohmu_k(C_2)$ in characteristic $2$).  Recall the direct strings and inverse strings are
  \[
  e_0,\ e_1,\ a,\ b,\ ab,\ a^{-},\ b^{-},\ b^{-}a^{-}.
  \] 
  To form homotopy strings we need reduced walks and thus in this small case we may consider only direct (or only inverse) homotopy letters. The composition of the homotopy letters must be a relation, so our homotopy strings (up to inverses) are
\[
(e_0),\ (e_1),\ (a),\ (b),\ (b)(a),\ (ab)^t,\ (ab)^t(a),\ (b)(ab)^t,\ (b)(ab)^t(a)
\]
for $t \geq 1$. There are no homotopy bands.  This gives a complete classification of the perfect complexes in the perfect derived category of the bound quiver. Thus this also gives a complete classification of the indecomposables in $\mathcal{D}^{\perf}(\cohmu_k(C_2))$ in characteristic two and equivalently the compact objects in the homotopy category of $C_2$-equivariant $H\und{k}$-modules.

Over the field $\F_2$, the complexes coming from the homotopy strings above correspond to the complexes $A_m$, $H(n)$, and $B_r$ for $m,r \geq 0$ and $n \in \Z$ from the classification in \cite{DHM}.  In particular, with $H=P_0$ and $F=P_1$, they are the complexes
\[
H(0),\ A_0,\ H(1),\ H(-1),\ B_0,\ A_t,\ H(t+1),\ H(-(t+1)),\ B_{t}
\]
respectively.
\end{example}

\section{Derived representation type of cohomological Mackey functors}\label[section]{sec:derived-wildnes-cohmackey}

We now turn to showing the cohomological Mackey algebra over a field of characteristic $p$ is derived wild in most cases.  We start with $G$ a $p$-group of order more than two and then generalize. In this section, we show the following.

\begin{theorem}\label[theorem]{thm:cohomological_derived_wild_surjection}
  Let $k$ be a field of characteristic $p$. The cohomological Mackey algebra  $\cohmu_k(G)$ is derived wild whenever $G$ surjects onto a $p$-group of order more than two.
\end{theorem}

Thus, using \cref{Thm:QuillenEquiv}, our main result in equivariant homotopy theory is that there is no meaningful classification of compact objects in the homotopy category of $G$-equivariant $H\und{k}$-modules.

\begin{theorem}\label[theorem]{thm:cohomological_derived_wild_surjection_topological}
 Let $k$ be a field of characteristic $p$. The category of compact objects in the homotopy category of $G$-equivariant $H\und{k}$-modules is wild whenever $G$ surjects onto a $p$-group of order more than two.
\end{theorem}

Over algebraically closed fields we have more tools available, and there we explore some tame cases in \cref{subsec:derived_tame_2m}.  Finally, we generalize the notion of semi-wild to the derived setting, which ultimately allows us to give the following complete classification of groups for which the cohomological Mackey algebra is derived wild.

\begin{theorem}\label[theorem]{thm:alg_closed_derived_type_cohomological_sylow}
Let $k=\bar{k}$ be an algebraically closed field of characteristic $p$ and let $G$ be a finite group.  Then $\cohmu_k(G)$ is derived wild if and only if the $p$-Sylow subgroup of $G$ has order more than two.
\end{theorem}

Applying \cref{Thm:QuillenEquiv} as usual, we have the following consequence.

\begin{theorem}\label[theorem]{thm:cohomological_derived_type_topological_algebraically_closed}
 Let $k=\bar{k}$ be an algebraically closed field of characteristic $p$. The category of compact objects in the homotopy category of $G$-equivariant $H\und{k}$-modules is wild if and only if the $p$-Sylow subgroup of $G$ has order more than two.
\end{theorem}

\subsection{\texorpdfstring{Derived wildness for $p$-groups}{Derived wildness for p-groups}}
In the previous section we saw that, over a field $k$ of characteristic $p$, the group algebra $k[C_p]$ is derived wild for $p$ odd.  We also saw that $k[C_p \times C_p]$ is derived wild for arbitrary primes. In this section, we will use this information to show the cohomological Mackey algebra $\cohmu_k(G)$ is derived wild for most $p$-groups. This strengthens a result of Webb \cite[Theorem~5.2]{Web24} showing indecomposable perfect complexes over $\cohmu_k(C_p)$ have arbitrarily large homology.

Recall from \cref{def:cohomological_mackey_algebra} that $\cohmu_k(G)$ is given by the opposite endomorphism ring of the transitive permutation modules $\hecke_k(G)$. This means that if $M$ is a direct sum of permutation modules, then $\End_{kG}(M)^{\op}$ is an idempotent subalgebra of $\cohmu_k(G)$. Using this together with \cref{prop:idempotent_subalgebra_derived_wild}, we prove the following.
\begin{theorem}\label[theorem]{thm:derived_wild_p-groups}
  Let $k$ be a field of characteristic $p$ and let $G$ be a nontrivial $p$-group not isomorphic to $C_2$. Then $\cohmu_k(G)$ is derived wild.
  \begin{proof}
    Consider first the case where $p$ is odd. Then there exists a normal subgroup $N$ such that $G/N \cong C_p$. Then $\End_{kG}(k[G/N])^{\op} \cong k[C_p]$ is an idempotent subalgebra of $\cohmu_k(G)$. Since the idempotent subalgebra $k[C_p]$ is derived wild by \cref{prop:wild_cyclic_group}, it follows from \cref{prop:idempotent_subalgebra_derived_wild} that $\cohmu_k(G)$ is derived wild as well.

    Consider now the case where $p=2$. Since $|G|>2$, we must have that $G$ maps onto a group of size $4$. We split this into two cases, $G/N \cong C_2 \times C_2$ or $G/N \cong C_4$.

    In the first case we get that $k[C_2 \times C_2]$ is an idempotent subalgebra, which is derived wild by \cref{prop:wild_CpxCp}. In the second case we get that $k[C_4]$ is an idempotent subalgebra, which is derived wild by \cref{prop:wild_cyclic_group}. In either case, we conclude that $\cohmu_k(G)$ is derived wild by \cref{prop:idempotent_subalgebra_derived_wild}.
  \end{proof}
\end{theorem}

A key step in the above proof is the following: for a normal subgroup $N$ of $G$, the group algebra $k[G/N]$ is an idempotent subalgebra of $\cohmu_k(G)$. In order to apply the same techniques above more generally, we separate this out as its own proposition. 
\begin{proposition}\label[proposition]{prop:G/N_idempotent_subalgebra_cohomological}
  Let $G$ be a finite group and  $N \trianglelefteq G$ be a normal subgroup. Then the following two statements hold.
  \begin{enumerate}
  \item If $\cohmu_k(G/N)$ is derived wild, then so is $\cohmu_k(G)$.
  \item If the group algebra $k[G/N]$ is derived wild, then so is $\cohmu_k(G/N)$.
  \end{enumerate}
  \begin{proof}
    Consider the sum of those permutation modules fixed by $N$:
    \begin{align*}
      M' \coloneqq \bigoplus_{N \leq H \leq G} k[G/H].
    \end{align*} 
    Then $\End_{kG}(M')^{\op} = \End_{kG/N}(M')^{\op} \cong \cohmu_k(G/N)$ is an idempotent subalgebra of $\cohmu_k(G)$.  Similarly, $\End_{kG}(k[G/N])^{\op} \cong k[G/N]$ is an idempotent subalgebra of $\cohmu_k(G/N)$. Hence, by \cref{prop:idempotent_subalgebra_derived_wild}, both results follow.
  \end{proof}
\end{proposition}

Putting together \cref{thm:derived_wild_p-groups} and \cref{prop:G/N_idempotent_subalgebra_cohomological}, we find a great deal of groups for which $\cohmu_k(G)$ is derived wild. For example, over a field of characteristic $p$, the cohomological Mackey algebra $\cohmu_k(G)$ is derived wild in each of the following cases:
\begin{enumerate}
  \item $G$ is a non-cyclic $p$-group
  \item $G = C_{p^m}$ for $p^m \geq 3$
  \item $G$ is an abelian group where $p^m$ divides the order of $G$ for $p^m \geq 3$
  \item $G$ has one of the above groups as quotient group.
\end{enumerate}

This completes the proof of our main result for the derived wildness of the cohomological Mackey algebra \cref{thm:cohomological_derived_wild_surjection}.

\bigskip

In order to address more general finite groups, we turn to the case where $k$ is algebraically closed for the rest of this section.

\subsection{\texorpdfstring{Derived tame for order $2m$}{Derived tame for order 2m}}\label[subsection]{subsec:derived_tame_2m}
Recall from \cref{thm:derived_dichotomy} of Bekkert and Drozd that, over an algebraically closed field $k = \bar{k}$, any finite dimensional $k$-algebra is either derived tame or derived wild, and never both. We now look at some cases where $\cohmu_k(G)$ is derived tame.

If the characteristic of $k$ does not divide the order of $G$, then Maschke's theorem tells us that the group algebra $kG$ is semisimple. Therefore every module is projective. 
In particular, the sum of the permutation modules $k[G/H]$ is a progenerator and so $\cohmu_k(G)$ is Morita equivalent to $kG$. Then every indecomposable perfect complex is just given by a stalk complex, and can be classified by the irreducible $kG$-modules, as in \cref{ex:semisimple_derived_finite}. Hence $\cohmu_k(G)$ is derived tame (in fact derived finite) whenever the characteristic of $k$ does not divide the order of $G$, as in \cref{cor:invertible_char_derived_finite}.

When $k$ has characteristic $2$, we have seen that $\cohmu_k(C_2)$ is a gentle algebra.  The indecomposable perfect complexes over a gentle algebra were classified in \cite{BM03} using homotopy string and homotopy bands. See \cref{ex:CohC2DerivedTameDHM} and \cref{ex:C2_gentle_homotopy_strings} for more details. In particular, when $k$ is algebraically closed, this means $\cohmu_k(C_2)$ is derived tame and hence not derived wild.  

We now cover one last family of examples that are derived tame.  All other cases will turn out to be derived wild.

\begin{proposition}\label[proposition]{prop:2m_derived_tame}
  Let $k=\bar{k}$ be an algebraically closed field of characteristic $2$.  Let $G$ be a group of order $2m$ for $m$ odd. Then $\cohmu_k(G)$ is derived tame.
  \begin{proof}
    By Burnside's normal p-complement theorem, $G$ is the semidirect product of $C_2$ acting on a group $H$ of order $m$. Then $kH$ will be semisimple and $kG$ will be the skew group algebra $kH \# C_2$. 
    Writing $kH$ as a product of matrix rings we can consider the $C_2$-action on the factors. 
    It will either interchange two factors or act on a single factor by automorphism. 
    This implies that $kG$ is the product of algebras of the form $M_n(k)^2 \# C_2$ and $M_n(k)\# C_2$, where the former action interchanges factors and the latter is arbitrary.

    In the first case we have $M_n(k)^2 \# C_2 \cong M_2(M_n(k)) = M_{2n}(k)$. In the second case, by the Skolem--Noether theorem, the $C_2$-action is given by conjugation by an element $u \in M_n(k)$ such that $u^2$ is in the center. Since the center is $k$ and $k$ is algebraically closed, we may assume $u^2 = 1$. Let $g$ be the generator of $C_2$.  Then $(1+ug)$ sits in the center of $M_n(k)\#C_2$ and satisfies $(1+ug)^2 = 0$.  Hence $M_n(k)\#C_2 = M_n(k)[x]/(x^2)$.

    This means that, up to Morita equivalence, $kG$ is the product of copies of $k$ and copies of $k[x]/(x^2)$. The only indecomposable $k[x]/(x^2)$-modules are the regular module and the simple module. Since $\cohmu_k(G)$ is the endomorphism ring of a $kG$-module, up to Morita equivalence it is the product of algebras each isomorphic to one of the following: $k$, $k[x]/(x^2)$, or $\End_{k[x]/(x^2)}(k[x]/(x^2) \oplus k[x]/(x))^{\op}$. All of these cases are gentle, hence derived tame by \cite{BM03}.
  \end{proof}
\end{proposition}

\subsection{Arbitrary finite groups}
In this section, we use Sylow subgroups of $G$ to determine the derived wildness of the cohomological Mackey algebra $\cohmu_k(G)$. 
As above, we restrict to the case where $k=\bar{k}$ is algebraically closed.
To show an algebra is derived wild in this setting, by \cref{thm:derived_dichotomy} it is enough to show it is not derived tame. At first this seems to be a difficult task.  
However, we can replace the definition of derived wild by a seemingly weaker notion, derived semi-wild. 

Inspired by the definition of derived wild from \cite{BD03} and \cite{Dro04}, and the definition of semi-wild representation type from \cite{Dro77}, we combine the two to make the following definition.

\begin{definition}[Derived semi-wild]
  Let $k= \bar{k}$ be an algebraically closed field.  We say that a $k$-algebra $\Lambda$ is \emph{derived semi-wild} if for any finite dimensional $k$-algebra $\Gamma$, there is a perfect complex of bimodules 
  \[
  P^\bullet \in K^b(\Lambda \otimes \Gamma^{\rm{op}}\mproj) \cong \mathcal{D}^{\perf}(\Lambda \otimes \Gamma^{\rm{op}})
  \]
  with the following property: for any $M$ in $\Gamma\mmod$, there exist (up to isomorphism) only finitely many $N$ with $P^\bullet \otimes_\Gamma M \cong P^\bullet \otimes_\Gamma N$. We call $P^\bullet$ a \emph{semi-strict family of $\Lambda$-complexes over $\Gamma$}.
\end{definition}

In fact, derived semi-wild implies derived wild for algebraically closed fields.

\begin{theorem}\label[theorem]{thm:semi_wild_equiv_wild}
  Let $k= \bar{k}$ be an algebraically closed field.  A $k$-algebra $\Lambda$ is derived semi-wild if and only if it is derived wild.
\end{theorem}

The argument for this equivalence uses dimension theory and a similar argument to \cite[{\usefont{T2A}{cmr}{m}{n}Предложение~2}]{Dro77}.  See also \cite[Theorem~27.10 (page~366)]{BSZ09} for the machinery involved and \cite[pages~6-8]{BD03} for the derived setting. 

We now wish to compare the cohomological Mackey algebra of $G$ with that of its $p$-Sylow subgroups. To do this we make use of the induction and restriction functors described in \cite[Section~3]{Web00}.

Recall that we can think of a cohomological Mackey functor as an additive functor on the Hecke category $\mathcal H_k(G)$. We may identify $\mathcal H_k(G)$ with its additive closure in $kG\mmod$. Let $\ind_H^G$ and $\res_H^G$ be the usual induction and restriction functors between $ kH\mmod$ and $kG\mmod$. Since these map permutation modules to permutation modules, they restrict to functors between $\mathcal H_k(H)$ and $\mathcal H_k(G)$. Thus we can make the following definition.

\begin{definition}
  Let $G$ be a finite group and $H \leq G$ a subgroup. If $M \colon \mathcal H_k(G) \to k\mmod$ and $N \colon \mathcal H_k(H) \to k\mmod$ are cohomological Mackey functors, we define the restriction and induction Mackey functors by
  \begin{align*}
    \res_{\cohmu_k(H)}^{\cohmu_k(G)} M &\coloneqq M \circ \ind_H^G\\
    \ind_{\cohmu_k(H)}^{\cohmu_k(G)} N &\coloneqq N \circ \res_H^G
  \end{align*}
\end{definition}

\begin{proposition}
  The restriction and induction of cohomological Mackey functors are both exact functors, and they map projective objects to projective objects.
  \begin{proof}
    Precomposition is always an exact functor, so exactness follows from the definition. If we again think of cohomological Mackey functors as functors on the Hecke category, then the projective objects are given by the representable functors. Since $\res_H^G$ and $\ind_H^G$ are mutual adjoints on both sides, we have 
    \begin{align*}
      \res^{\cohmu G}_{\cohmu H} \Hom_{kG}(P,-) &= \Hom_{kG}(P,\ind_H^G-) = \Hom_{kH}(\res_H^G P,-)\\
      \ind^{\cohmu G}_{\cohmu H} \Hom_{kH}(Q,-) &= \Hom_{kH}(Q,\res_H^G-) = \Hom_{kG}(\ind_H^G Q,-).
    \end{align*}
    Since $\res_H^G$ and $\ind_H^G$ send permutation modules to permutation modules, these are again projective.
  \end{proof}
\end{proposition}

We now prove a well-known relationship between induction and restriction that is crucial for our reduction.

\begin{proposition}
  For $H \leq G$ finite groups, the identity on $kH\mmod$ is a direct summand of $\res_H^G \circ \ind_H^G$. If $[G:H]$ is invertible in $k$, then the identity on $kG\mmod$ is a direct summand of $\ind_H^G \circ \res_H^G$.
  \begin{proof}
    As a map of $kH$-modules, the inclusion $kH \to kG$ splits, simply by sending $g$ to $0$ for $g$ not in $H$. Hence $\res_H^G \circ \ind_H^G = kG \otimes_{kH}-$ has $kH\otimes_{kH}-$ as a direct summand, which is the identity on $kH\mmod$.

    For $X \in kG\mmod$ we have that $g \otimes x \mapsto gx$ is a surjective natural transformation from $kG \otimes_{kH} X = \ind_H^G \circ \res_H^G X$ to $X$. When $[G:H]$ is invertible in $k$ this has a natural splitting given by
    \begin{align*}
      x \mapsto \frac{1}{[G:H]}\sum_{g_i} g_i \otimes g_i^{-1} x
    \end{align*}
    where $g_i$ runs over representatives for $G/H$.
  \end{proof}
\end{proposition}

The following technique is inspired by \cite{BD77}.

\begin{theorem}
  Let $k=\bar{k}$ be an algebraically closed field of characteristic $p$,
   $G$ a finite group, and $H \leq G$ its $p$-Sylow subgroup. Then $\cohmu_k(G)$ is derived semi-wild if and only if $\cohmu_k(H)$ is.
  \begin{proof}
    Let us assume that $\cohmu_k(G)$ is derived semi-wild. Then for every finite dimensional algebra $\Gamma$ we can find a semi-strict family of complexes $P^\bullet$. Since the restriction functor is exact and preserves projective objects, restricting $P^\bullet$ we can define the complex $Q^\bullet \coloneqq \res_{\cohmu_k(H)}^{\cohmu_k(G)}(P^{\bullet})$, which is a perfect complex of $\cohmu_k(H)$--$\Gamma$-bimodules.

    Consider two $\Gamma$-modules $M$ and $N$ such that $Q^\bullet \otimes M \cong Q^\bullet \otimes N$. Then 
    \begin{align*}
      \ind_{\cohmu_k(H)}^{\cohmu_k(G)} Q^\bullet \otimes M = \ind_{\cohmu_k(H)}^{\cohmu_k(G)}\res_{\cohmu_k(H)}^{\cohmu_k(G)} P^\bullet \otimes M \cong \ind_{\cohmu_k(H)}^{\cohmu_k(G)}\res_{\cohmu_k(H)}^{\cohmu_k(G)} P^\bullet \otimes N.
    \end{align*}
    Since $P^\bullet \otimes N$ is a direct summand of 
    \[
    \ind_{\cohmu_k(H)}^{\cohmu_k(G)}\res_{\cohmu_k(H)}^{\cohmu_k(G)} P^\bullet \otimes N \cong \ind_{\cohmu_k(H)}^{\cohmu_k(G)}\res_{\cohmu_k(H)}^{\cohmu_k(G)} P^\bullet \otimes M,
    \] 
    which has finitely many direct summands up to isomorphism, there are only finitely many possibilities for what $P^\bullet \otimes N$ can be. Combining this with the fact that $P^\bullet$ is a semi-strict family, it follows that $Q^\bullet$ is also a semi-strict family.

    The proof of the converse is exactly the same, replacing the roles of restriction and induction, and using that $[G:H]$ is relatively prime to $p$.
  \end{proof}
\end{theorem}

The above theorem taken together with \cref{thm:derived_wild_p-groups}, \cref{prop:2m_derived_tame}, and \cref{thm:semi_wild_equiv_wild}, proves the classification of the derived representation type of cohomological Mackey functors, \cref{thm:alg_closed_derived_type_cohomological_sylow}.

\section{The Mackey algebra}\label[section]{sec:mackey_algebra}
In this section, we collect some results about the larger Mackey algebra.  Most of these follow fairly immediately from work in previous sections.  The main exception is the derived representation type of $\mu_k(C_2)$.

\subsection{Representation type of the Mackey algebra}
Here we use the results of \cref{sec:rep_type_cohom_mackey} to deduce some consequences for the representation type of the Mackey algebra over $k$ a field of characteristic $p$.

Recall from \cite[Theorem~18.1]{TW95} (see also \cref{thm:finite-rep-type-of-mackey-algebras}) the Mackey algebra $\mu_k(G)$ has finite representation type if and only if $p^2 \! \not| \ |G|$. 
Thus, for $k$ a field of characteristic $2$, the Mackey algebra $\mu_k(C_4)$ is representation infinite.  We showed in \cref{prop:C4_tame} that over $k = \bar{k}$, the cohomological Mackey algebra $\cohmu_k(C_4)$ is tame. We do not pursue the question of representation type for $\mu_k(C_4)$ here.

In \cref{sec:rep_type_cohom_mackey}, we largely focused on cyclic groups.  
There we showed the cohomological Mackey algebra has wild representation type whenever $G=C_{p^m}$ with $m \geq 2$ and $p^m > 4$.  
As an immediate corollary to \cref{prop:cohmu_cyclic_more_than_9_wild} and \cref{prop:cohmu_C8_C9_wild}, the cohomological Mackey algebra being wild makes the Mackey algebra wild.

\begin{corollary}\label[corollary]{cor:mackey_algebra_wild_cyclic}
  Let $k$ be a field of characteristic $p$. The Mackey algebra $\mu_k(C_{p^m})$ is wild when $m \geq 2$ and $p^m > 4$.
\end{corollary}

\begin{proof}
  The Mackey algebra $\mu_k(G)$ surjects onto the cohomological Mackey algebra $\cohmu_k(G)$ and the restriction of scalars is a representation embedding.
\end{proof}

\subsection{Derived representation type of the Mackey algebra}
Here we look at the derived representation type of the Mackey algebra over $k$ a field of characteristic $p$.

The derived wild group algebras from \cref{sec:derived_wild_groups} have some immediate consequences for the derived representation type of the Mackey algebra.

\begin{theorem}\label[theorem]{thm:derived_wild_p-groups_mackey}
  Let $k$ be a field of characteristic $p$.  Let $G$ be either
  \begin{enumerate}
  \item $G=C_{p^m}$ for $p^m > 2$, or
  \item $G=C_p \times C_p$.
  \end{enumerate}
  Then the Mackey algebra $\mu_k(G)$ is derived wild.
\end{theorem}

\begin{proof}
  The group algebra $kG$ is an idempotent subalgebra of $\mu_k(G)$, so this follows from \cref{prop:wild_cyclic_group} and \cref{prop:wild_CpxCp}.
\end{proof}

In fact, the argument for the derived wildness of $k[c]/(c^n)$ generalizes fairly immediately to cover a large class of algebras, which include all characteristic $p$ group algebras for finite $p$-groups of order more than two.

\begin{proposition}
  Let $k$ be a field and let $R$ be a finite dimensional local Frobenius $k$-algebra with length more than two. Then $R$ is derived wild.
\end{proposition}

\begin{proof}
  The proof is almost identical to the proof of \cref{prop:wild_cyclic_group}.  Throughout the argument, replace $c^n$ with an element $x$ of the socle and replace $c^{n-1}$ with an element $y$ in the radical of $R$ that is not in the socle.
\end{proof}

As a consequence, the Mackey algebra is derived wild for $p$-groups of order more than two.

\begin{theorem}
  Let $k$ be a field of characteristic $p$ and let $G$ be a finite $p$-group of order more than two. Then the Mackey algebra $\mu_k(G)$ is derived wild.
\end{theorem}

One might hope that, like the cohomological version $\cohmu_k(C_2)$, the Mackey algebra $\mu_k(C_2)$ would be derived tame.  However, it is also derived wild.  The proof cannot rely on the group algebra, so we include a direct argument adapted from the argument in \cite{BDF09}.

\begin{proposition}
  The quiver with relations 
  \begin{center}
    $\Lambda \coloneqq
    \left.
    k\left[
    \begin{tikzcd}
      1 \ar[r, bend left, "b"] & 0 \ar[l, bend left, "a"]
    \end{tikzcd}
    \right]
    \middle/ 
    \langle aba,\ bab \rangle,
    \right.$
  \end{center}
  which describes the Mackey algebra $\mu_k(C_2)$ over a field of characteristic $2$ (see \cref{ex:C2characteristic2MackeyQuiver}), is derived wild. 
  \begin{proof}
    Recall from \cref{prop:A1tildetilde_wild}, that the quiver below is wild.
    \begin{center}
      \begin{tikzcd}
        1 \ar[r, bend left, "x"] \ar[r, bend right, swap, "y"] & 2 \ar[r, "w"] & 3
      \end{tikzcd}
    \end{center}
    Let $\Gamma$ be the path algebra of this wild quiver and write ${}_iP_j \coloneqq \Lambda e_i \otimes e_j\Gamma$. Consider the following complex of $R$--$\Gamma$-bimodules. 
    \begin{center}
      \begin{tikzcd}[ampersand replacement = \&, column sep=3em, row sep = 0em]
        \& {}_1P_1 \ar[dr]{}{\makebox[0pt]{$\begin{pmatrix}-a\\a\end{pmatrix}$}} 
        \&\&\&
        {}_0P_1 \ar[dr]{}{\makebox[0pt]{$b \otimes x$}}\\
        {}_1P_1 \ar[ur]{}{\makebox[0pt]{$ab$}} \ar[dr, swap]{}{\makebox[0pt]{$a$}} 
        \& \oplus \& 
        {{}_0P_1}^2 \ar[r]{}{\makebox[0pt]{$\begin{pmatrix}ba&0\\ 0 & ba \end{pmatrix}$}} 
        \&
        {{}_0P_1}^2 \ar[ur]{}{\makebox[0pt]{$\begin{pmatrix}ba&0\end{pmatrix}$}} 
        \ar[dr, swap]{}{\makebox[0pt]{$\begin{pmatrix}0&b\end{pmatrix}$}} 
        \& \oplus \& 
        {}_1P_2 \ar[r]{}{\makebox[0pt]{$a \otimes w$}} \& {}_1P_3\\
        \& {}_0P_1 \ar[ur, swap]{}{\makebox[0pt]{$\begin{pmatrix}ba\\0\end{pmatrix}$}} 
        \&\&\&
        {}_1P_1 \ar[ur, swap]{}{\makebox[0pt]{$ab \otimes y$}} 
      \end{tikzcd}
    \end{center}
    Using similar reasoning as in the proof of \cref{thm:derived_wild_p-groups}, one sees that this is a strict family of $\Lambda$-complexes over $\Gamma$.
  \end{proof}
\end{proposition}

Collecting these results, we have shown the Mackey algebra is derived wild for cyclic $p$-groups.

\begin{theorem}\label[theorem]{thm:derived_wild_more_p-groups_mackey}
  Let $k$ be a field of characteristic $p$ and let $G$ be a nontrivial finite $p$-group. Then the Mackey algebra $\mu_k(G)$ is derived wild.
\end{theorem}

Applying \cref{Thm:QuillenEquiv}, there is also no meaningful classification of compact objects in the homotopy category of $H\und{A}_k$-modules.

\begin{theorem}\label[theorem]{thm:derived_wild_more_p-groups_mackey_topological}
  Let $k$ be a field of characteristic $p$.  The category of compact objects in the homotopy category of $G$-equivariant $H\und{A}_k$-modules is wild whenever $G$ is a nontrivial finite $p$-group.
\end{theorem}

\section{Singularity categories}\label[section]{sec:singularity_category}
In this section, we look at the singularity category of the cohomological Mackey algebra for $G$ a cyclic $p$-group and for $G=C_2 \times C_2$.

The singularity category, given by the Verdier quotient 
\[
\mathcal{D}_{\sg}(R) = \mathcal{D}^b(R)/\mathcal{D}^{\perf}(R),
\]
is a derived invariant of a Noetherian ring $R$.  It measures the failure of $R$ to have finite global dimension. In general, the singularity category can be quite hard to compute, but for some families of rings we have reliable methods.

An algebra is called (Iwanaga--)Gorenstein if all projectives have finite injective dimension and all injectives have finite projective dimension. For Gorenstein algebras there is an equivalence of triangulated categories 
\begin{align*}
 \mathcal{D}_{\sg}(R) \cong \underline{\CM}(R)
\end{align*}
between the singularity category and the stable category of Cohen--Macaulay modules \cite{Hap91, Buc21}, which is more accessible to computation.

For a Gorenstein algebra $R$, the category of Cohen--Macaulay modules is the full subcategory of modules without extensions to projectives,
\begin{align*}
  \CM(R) \coloneqq \{ X \in R\mmod \mid \Ext^i(X, R) = 0,\ i>0 \}.
\end{align*}
The stable category $\underline{\CM}(R)$ is the quotient by the ideal of morphisms factoring through a projective module.

There are different ways to compute Cohen--Macaulay modules, so for Gorenstein algebras we have a more reliable method of computing the singularity category.  It is known precisely when the cohomological Mackey algebra is Gorenstein.

\begin{theorem}\cite[Theorem~5.3]{BSW17}
The cohomological Mackey algebra $\cohmu_k(G)$ over a field of characteristic $p$ is Gorenstein if and only if $G$ satisfies one of the following:
\begin{enumerate}
  \item the $p$-Sylow subgroup of $G$ is cyclic, or
  \item $p=2$ and the $2$-Sylow subgroup is dihederal.
\end{enumerate} 
\end{theorem}   
In this section we investigate a few such examples using precluster tilting modules and relative homological algebra.

\subsection{Singularity category for the cohomological Mackey algebra of $C_p$}
A module $M$ is called \emph{precluster tilting} if it is generating-cogenerating, and $\tau M$ and $\tau^- M$ are in $\add M$, where $\tau$ is the Auslander--Reiten translate and $\add M$ is the full subcategory of direct summands of $M^n$.

Let $M$ be a precluster tilting module over $\Lambda$ and let $\Gamma = \End(M)^{\op}$. Then $\Gamma$ is Gorenstein \cite{IS18}.  Moreover, we have an equivalence of categories
\begin{align*}
  \Hom(M,-)\colon \Lambda\mmod \to \CM(\Gamma),
\end{align*}
which restricts to an equivalence between $\add M$ and $\add \Gamma$. This equivalence induces an equivalence of categories between $\Lambda\mmod / [M]$ and $\underline{\CM}(\Gamma)$, where $[M]$ is the ideal of morphisms that factor through $\add M$ \cite{IS18}.

The singularity category is also a triangulated category. Through the equivalence $\mathcal{D}_{\sg}(\Gamma) \cong \underline{\CM}(\Gamma)$, the inverse shift functor is given by taking syzygies in $\Gamma\mmod$. Through the equivalence $\Lambda\mmod \cong \CM(\Gamma)$ syzygies can be computed as kernels of $\add M$-approximations.

In the following example we compute the singularity category of a family of examples, which will include $\cohmu_k(C_p)$.

\begin{example}\label[example]{ex:Cp_singularity_cat}
  Let $k$ be a field and let $\Lambda = k[t]/(t^{n})$. Note that $\Lambda$ is self-injective and so generating implies generating-cogenerating.  In this case, $\tau$ is the identity on nonprojective indecomposable objects. Thus any generating module is precluster tilting. 
  
  The AR-quiver of $\Lambda$ can be described as 
  \[\begin{tikzcd}
    1 & 2 & 3 & \cdots & {n-1} & n \phantom{-1}
    \arrow["{x_1}", curve={height=-6pt}, from=1-1, to=1-2]
    \arrow["{y_1}", curve={height=-6pt}, from=1-2, to=1-1]
    \arrow["{x_2}", curve={height=-6pt}, from=1-2, to=1-3]
    \arrow["{y_2}", curve={height=-6pt}, from=1-3, to=1-2]
    \arrow["{x_3}", curve={height=-6pt}, from=1-3, to=1-4]
    \arrow["{y_3}", curve={height=-6pt}, from=1-4, to=1-3]
    \arrow["{x_{n-2}}", curve={height=-6pt}, from=1-4, to=1-5]
    \arrow["{y_{n-2}}", curve={height=-6pt}, from=1-5, to=1-4]
    \arrow["{x_{n-1}}", curve={height=-6pt}, from=1-5, to=1-6]
    \arrow["{y_{n-1}}", curve={height=-6pt}, from=1-6, to=1-5]
  \end{tikzcd}\]
  such that $y_1 x_1 = 0$ and $y_{i+1}x_{i+1} = x_{i}y_i$ for $1\leq i \leq n-2$.  Here the vertex $i$ represents the unique $i$-dimensional indecomposable module.

  For any module $M$, we can recover the AR-quiver of 
  \[
  \Lambda\mmod / [M] \cong \underline{\CM}(\Gamma) \cong \mathcal{D}_{\sg}(\Gamma)
  \]
  by killing all paths that factor through a summand of $M$.  For example, take the module $M = \Lambda \oplus \Lambda/(t)$, then $\mathcal{D}_{\sg}(\End(M)^{\op})$ is the $k$-linear category given by 
  \[\begin{tikzcd}
    {\color{gray}1} & 2 & 3 & \cdots & {n-1} & {\color{gray}n}\phantom{-1}
    \arrow["{x_1}", dashed, gray, curve={height=-6pt}, from=1-1, to=1-2]
    \arrow["{y_1}", dashed, gray, curve={height=-6pt}, from=1-2, to=1-1]
    \arrow["{x_2}", curve={height=-6pt}, from=1-2, to=1-3]
    \arrow["{y_2}", curve={height=-6pt}, from=1-3, to=1-2]
    \arrow["{x_3}", curve={height=-6pt}, from=1-3, to=1-4]
    \arrow["{y_3}", curve={height=-6pt}, from=1-4, to=1-3]
    \arrow["{x_{n-2}}", curve={height=-6pt}, from=1-4, to=1-5]
    \arrow["{y_{n-2}}", curve={height=-6pt}, from=1-5, to=1-4]
    \arrow["{x_{n-1}}", dashed, gray, curve={height=-6pt}, from=1-5, to=1-6]
    \arrow["{y_{n-1}}", dashed, gray, curve={height=-6pt}, from=1-6, to=1-5]
  \end{tikzcd}\]
  with the relations as above, except now the objects $1$ and $n$ are $0$ (shown in gray).
\end{example}

We can apply this example to compute the singularity category of $\cohmu_k(C_p)$

\begin{theorem}
  Let $k$ be a field of characteristic $p$ and let $G = C_{p^m}$.  The singularity category of the cohomological Mackey algebra $\cohmu_k(C_{p^m})$ is the triangulated category given by the product of $m$ factors of categories of the form
    \[\begin{tikzcd}
    1 & 2 & 3 & \cdots & {n-1} & n \phantom{-1}
    \arrow["{x_1}", curve={height=-6pt}, from=1-1, to=1-2]
    \arrow["{y_1}", curve={height=-6pt}, from=1-2, to=1-1]
    \arrow["{x_2}", curve={height=-6pt}, from=1-2, to=1-3]
    \arrow["{y_2}", curve={height=-6pt}, from=1-3, to=1-2]
    \arrow["{x_3}", curve={height=-6pt}, from=1-3, to=1-4]
    \arrow["{y_3}", curve={height=-6pt}, from=1-4, to=1-3]
    \arrow["{x_{n-2}}", curve={height=-6pt}, from=1-4, to=1-5]
    \arrow["{y_{n-2}}", curve={height=-6pt}, from=1-5, to=1-4]
    \arrow["{x_{n-1}}", curve={height=-6pt}, from=1-5, to=1-6]
    \arrow["{y_{n-1}}", curve={height=-6pt}, from=1-6, to=1-5]
  \end{tikzcd}\]
  for $n = p^{j} - p^{j-1}-1$ and $1 \leq j \leq m$, together with relations $y_1 x_1 = 0$, $x_{n-1}y_{n-1}=0$ and $y_{i+1}x_{i+1} = x_{i}y_i$ for $1\leq i \leq n-2$.  The shift is given by flipping each of the diagrams above.
\end{theorem}
\begin{proof}
  Notice that when $k$ has characteristic $p$ and $n=p$, the \cref{ex:Cp_singularity_cat} computes the singularity category of the Hecke algebra $\End(M)^{\op} = \mathcal{E}_k(C_p)$, which we can identify with $\cohmu_k(C_p)$ using Yoshida's characterization. More generally, the singularity category of $\cohmu_k(C_{p^m})$ can be described in a similar way, taking $n=p^m$ and deleting every $p$th-power vertex. 
\end{proof}

\subsection{Singularity category for the Klein four cohomological Mackey algebra}
We now turn our attention to $\cohmu_k(C_2 \times C_2)$. In this case, the sum of transitive permutation modules does not define a precluster tilting module over the group algebra, so we cannot directly use the same approach as before. However, we can do something similar with the help of some tools from relative homological algebra, which replaces the exact structure of the module category with a relative exact structure. We briefly recall some standard constructions, see \cite{AS93} for more details.

Given a module $M$ over a finite dimensional $k$-algebra $\Lambda$, one can define two exact structures $F_M$ and $F^M$ on $\Lambda\mmod$, forcing $M$ to be either projective or injective as below \cite[Prop~1.7]{AS93}.
\begin{align*}
  \Ext^1_{F_M}(C, A) &\coloneqq \{ \eta \in \Ext^1(C, A) | \Hom(M, \eta) \text{ is exact} \} \\
  \Ext^1_{F^M}(C, A) &\coloneqq \{ \eta \in \Ext^1(C, A) | \Hom(\eta, M) \text{ is exact} \}
\end{align*}

In the $F_M$ exact structure, the $F_M$-relative indecomposable projectives are the usual indecomposable projectives of $\Lambda$ together with the summands of $M$. Thus the $F_M$-relative projective modules are $\add M \oplus \Lambda$. The $F_M$-relative injective modules are $\add \tau M \oplus D\Lambda$ and $F_M = F^{\tau M}$ \cite[Prop~1.8,~1.10]{AS93}.

If $M$ is a generator-cogenerator over a finite dimensional algebra $\Lambda$ and $\Gamma = \End(M)^{\op}$, then there is an equivalence of categories between modules and second syzygies called the Morita--Tachikawa correspondence given by
\begin{align*}
  \Hom(M, -) \colon \Lambda \mmod\to \Omega^2 \Gamma\mmod.
\end{align*}
Moreover, it restricts to an equivalence between $\add M$ and $\add \Gamma$, and 
\begin{align*}
  \Ext_{F_M}^\bullet(X, Y) = \Ext_\Gamma^\bullet((M, X), (M, Y)).
\end{align*}
This means that when $\Gamma$ is Gorenstein, we can compute $\CM(\Gamma)$ as the $F_M$-orthogonal complement to $M$ in $\Lambda\mmod$.

Let $k$ be a field of characteristic $2$ and let $$\Lambda = k[a,b]/(a^2, b^2) \cong k[C_2 \times C_2].$$ 
Then $\Lambda$ is special biserial and we can describe the indecomposable modules by strings and bands \cite{WW85,CB18}. If $\omega$ is a string we write $M(\omega)$ for the associated string module, and if $\omega$ is a band and $A$ an indecomposable invertible matrix we write $M(\omega, A)$ for the band module in which $A$ describes the action of moving around the band. The strings of $\Lambda$ are up to inverse
\[
e_1,\ a,\ b,\ (ab^-)^i,\ (b^-a)^i,\ (ab^-)^ia, \quad \text{and} \quad (b^-a)^ib^- \quad\text{for } 1\leq i <\infty.
\] 
The only band is $ab^-$. Some of the strings are depicted in \cref{fig:strings_over_klein4}.  See also \cite[Section~11.5]{WebbCourseNotes} for another description of the indecomposables of $k[C_2 \times C_2]$.
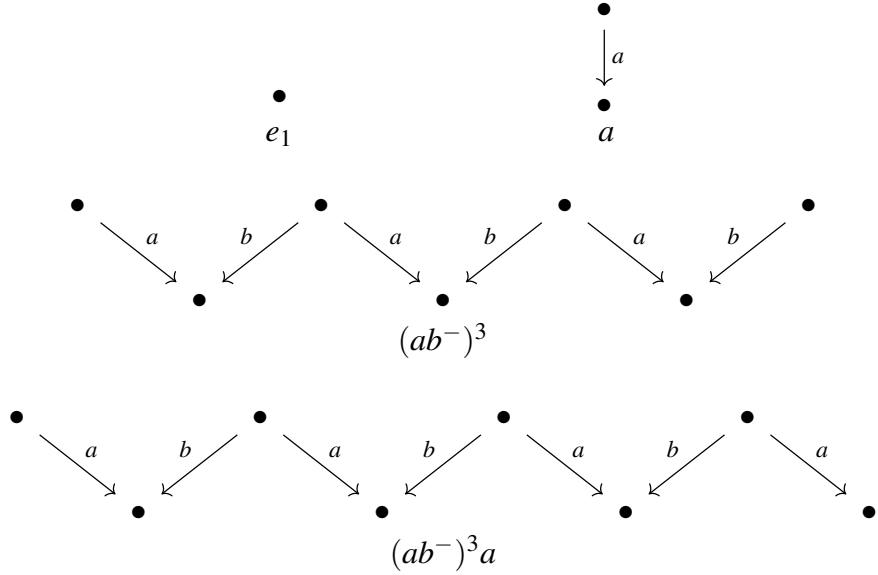
\begin{figure}
  \centering
  \begin{subfigure}{0.3\textwidth}
    \centering
    \begin{tikzcd}
      \bullet
    \end{tikzcd}\\
    $e_1$
  \end{subfigure}
  \begin{subfigure}{0.3\textwidth}
    \centering
    \begin{tikzcd}
      \bullet \ar[d, "a"]\\
      \bullet
    \end{tikzcd}\\ 
    $a$
  \end{subfigure}\vspace{0.5cm}

  \begin{subfigure}{\textwidth}
    \centering
    \begin{tikzcd}
      \bullet \ar[dr, "a"] && \bullet \ar[dl, swap, "b"]\ar[dr, "a"] && \bullet\ar[dl, swap, "b"]\ar[dr, "a"] && \bullet \ar[dl, swap, "b"]\\
      &\bullet && \bullet &{}&{}\bullet
    \end{tikzcd}\\
    $(ab^-)^3$
  \end{subfigure}\vspace{0.5cm}

  \begin{subfigure}{\textwidth}
    \centering
    \begin{tikzcd}
      \bullet \ar[dr, "a"] && \bullet \ar[dl, swap, "b"]\ar[dr, "a"] && \bullet\ar[dl, swap, "b"]\ar[dr, "a"] && \bullet \ar[dl, swap, "b"]\ar[dr, "a"]\\
      &\bullet && \bullet &{}&\bullet&& \bullet
    \end{tikzcd}\\
    $(ab^-)^3a$
  \end{subfigure}
  \caption{Some strings over $k[C_2 \times C_2]$}\label[figure]{fig:strings_over_klein4}
\end{figure}

The transitive permutation modules are given by: 
\[
\Lambda \text{ itself},\ \ M_a \coloneqq M(a),\ \ M_b \coloneqq M(b),\ \ M_{ab} \coloneqq M(ab^{-}, 1),\quad \text{and} \quad S \coloneqq M(e_1).
\] 
Letting $M$ denote the direct sum of all of these, we have $\cohmu_k(C_2 \times C_2) \cong \End(M)^{\op}$. Now using $F_M$, the relative exact structure making $M$ projective, our goal is to compute the $F_M$-orthogonal complement to $M$ in $\Lambda\mmod$. That is, we want to find all $X$ such $\Ext_{F_M}^*(X,M)=0$. These will correspond to the Cohen--Macaulay modules.

We have $\Lambda = D\Lambda$, as well as $\tau M_* = M_*$ for $*=a,\, b,\, ab$ and $\tau S = M(ab^-ab^-)$. Thus $\Lambda$ and the modules $M_*$ are all $F_M$-injective. So we need only worry about modules that have nontrivial $F_M$-extensions with $S$. 

An $F_M$-injective resolution of the simple module $S$ is given by 
\begin{center}
  \begin{tikzcd}
    0 \ar[r] & S \ar[r] & M_a \oplus M_b \oplus M_{ab} \ar[r] & \tau S \ar[r] & 0.
  \end{tikzcd}
\end{center}
As usual, a module $X$ is $F_M$-orthogonal if and only if the above sequence remains exact when we apply the functor $\Hom(X,-)$. So for each indecomposable module $X$, we compute the dimension of $\Hom(X, -)$ on $S$, $M_*$ and $\tau S$ by brute force. The results are displayed in \cref{tab:hom-dimensions}. Note $J_n(1)$ denotes the $n\times n$ Jordan block with eigenvalue $1$ and $A$ denotes an indecomposable $n\times n$ matrix without $1$ as an eigenvalue.

\begin{table}[h]
  \begin{tabular}{|l||c|c|c|c|c|}
  \hline
   & $S$ & $M_a$ & $M_b$ & $M_{ab}$ & $\tau S$ \\ \hline\hline
   $M((ab^-)^n)$     & $n+1$ & $n+1$ & $n+1$ & $n+1$ & $3n+1$ \\ \hline
   $M((b^-a)^n)$     & $n$   & $n+1$ & $n+1$ & $n+1$ & $3n+2$ \\ \hline
   $M((ab^-)^na)$    & $n+1$ & $n+2$ & $n+1$ & $n+1$ & $3n+3$ \\ \hline
   $M((b^-a)^nb^-)$  & $n+1$ & $n+1$ & $n+2$ & $n+1$ & $3n+3$ \\ \hline
   $M(ab^-, J_n(1))$ & $n$   & $n$   & $n$   & $n+1$ & $3n$ \\ \hline
   $M(ab^-, A)$      & $n$   & $n$   & $n$   & $n$   & $3n$ \\ \hline
  \end{tabular}
  \caption{Dimension of $\Hom(X, Y)$}\label[table]{tab:hom-dimensions}
\end{table}

Using \cref{tab:hom-dimensions} and the formula 
\begin{align*}
  \dim \Ext^1_{F_M}(X, S) = \dim \Hom(X, S\oplus \tau S) - \dim\Hom(X, M_a \oplus M_b \oplus M_{ab})
\end{align*}
we can compute the dimensions of the extension groups $\Ext^1_{F_M}(X, S)$ for each indecomposable $X$. 
The results are shown in \cref{tab:ext-dimensions}

\begin{table}[h]
  \begin{tabular}{|l||c|}
  \hline
   $M((ab^-)^n)$     & $n-1$ \\ \hline
   $M((b^-a)^n)$     & $n-1$ \\ \hline
   $M((ab^-)^na)$    & $n$   \\ \hline
   $M((b^-a)^nb^-)$  & $n$   \\ \hline
   $M(ab^-, J_n(1))$ & $n-1$ \\ \hline
   $M(ab^-, A)$      & $n$   \\ \hline
  \end{tabular}
   \caption{Dimension of $\Ext^1_{F_M}(X, S)$}\label[table]{tab:ext-dimensions}
\end{table}

The injective resolution of $S$ is short so there are no higher $\Ext$ groups. We see from \cref{tab:ext-dimensions} that besides the summands of $M$, the only indecomposables that are $F_M$-orthogonal to $M$ are $M(ab^-)$ and $M(b^- a)$ (since $M(ab^-,J_n(1))$ at $n=1$ is $M(ab^-,1) = M_{ab}$). Thus these correspond to the Cohen--Macaulay modules of $\cohmu_k(C_2 \times C_2)$.

Now to compute the singularity category of $\cohmu_k(C_2 \times C_2)$, we take the stable category of Cohen--Macaulay modules. Hence the singularity category of $\cohmu_k(C_2 \times C_2)$ will have two indecomposable objects corresponding $M(ab^-)$ and $M(b^- a)$. For the morphisms, we observe that any homomorphism between these two must factor through $M$ and similarly for any radical endomorphism. So there are no nontrivial morphisms in the singularity category and we conclude the following.

\begin{theorem}
  Let $k$ be a field of characteristic $2$. The singularity category of the cohomological Mackey algebra $\cohmu_k(C_2 \times C_2)$ is the triangulated category $k^2\mmod$, which is semisimple. The triangulated shift functor is given by swapping the two indecomposable objects.
\end{theorem}

\bibliography{./wildbib}
\bibliographystyle{alpha}
\end{document}